\documentclass[11pt]{article}
\usepackage{amssymb}
\usepackage{mathrsfs}
\usepackage{latexsym,amsmath,amssymb,amsfonts,epsfig,graphicx,cite,psfrag}
\usepackage{eepic,color,colordvi,amscd}
\usepackage{ebezier}
\usepackage{verbatim}
\usepackage{subfigure}
\usepackage{enumitem}
\usepackage{algorithm}
\usepackage{algorithmic}
\usepackage{hyperref}
\usepackage{amsthm}
\usepackage{longtable}
\usepackage{comment}
\usepackage{color}
\usepackage{graphicx} 
\usepackage{epstopdf}
\definecolor{blue}{rgb}{0,0,0.9}
\definecolor{red}{rgb}{0.9,0,0}
\definecolor{green}{rgb}{0,0.9,0}

\theoremstyle{plain}
\newtheorem{theo}{Theorem}[section]

\newtheorem{lem}[theo]{Lemma}
\newtheorem{coro}[theo]{Corollary}

\newtheorem{prop}[theo]{Proposition}

\newtheorem{exmp}[theo]{Example}

\theoremstyle{definition}

\theoremstyle{remark}
\newtheorem{rem}[theo]{Remark}

\def\<{\langle}
\def\>{\rangle}
\def\dist{{\rm dist}}
\def\E{\mathcal{E}}

\def\A{\mathcal{A}}
\def\L{\mathcal{L}}
\def\F{\mathcal{F}}

\def\R{\mathbb{R}}
\def\({\left(}
\def\){\right)}
\def\prox{\textbf{Prox}_g}
\def\ll{\|}
\def\prox1{\textbf{Prox}}

\let\svthefootnote\thefootnote
\newcommand\blankfootnote[1]{%
	\let\thefootnote\relax\footnotetext{#1}%
	\let\thefootnote\svthefootnote%
}

\def\inprod#1#2{\langle #1,\,#2\rangle}
\def\norm#1{\|#1\|}
\def\gam{\gamma}
\def\lam{\lambda}

\begin{document}
	\title{Self-adaptive ADMM for semi-strongly convex problems}
		
	\author{Tianyun Tang
\thanks{Department of Mathematics, National
         University of Singapore, Singapore
         119076 ({\tt ttang@u.nus.edu}).
         }, \quad 
	 Kim-Chuan Toh\thanks{Department of Mathematics, and Institute of 
Operations Research and Analytics, National
         University of Singapore, 
       Singapore
         119076 ({\tt mattohkc@nus.edu.sg}).  The research of this author is supported
by the Ministry of Education, Singapore, under its Academic Research Fund Tier 3 grant call (MOE-2019-T3-1-010).}
	 }
	
	\date{\today}
	
	\maketitle

\begin{abstract}
In this paper, we develop a self-adaptive ADMM that updates the penalty parameter adaptively. When one part of the objective function is strongly convex i.e., the problem is semi-strongly convex, our algorithm can update the penalty parameter adaptively with guaranteed convergence. { We establish various types of convergence results including accelerated convergence rate of $O(1/k^2),$ linear convergence and convergence of iteration points. This enhances various previous results because we allow the penalty parameter to change adaptively. We also develop a partial proximal point method with the subproblem solved by our adaptive ADMM. This enables us to solve problems without semi-strongly convex property.} Numerical experiments are conducted to demonstrate the high 
efficiency and robustness of our method.
\end{abstract}

\bigskip
\noindent{\bf keywords:} Adaptive ADMM, Semi-strongly convex, Partial proximal point method
\\[5pt]

\section{Introduction}
\subsection{Adaptive ADMM}
In this paper, we consider the following linearly constrained convex optimization problem 
\begin{equation}
\min \left\{ f(y)+g(z)\,\mid\ By+Cz=b \right\},
\label{eq-prob}
\end{equation}
where $B\in \R^{m\times n_1}$, $C\in \R^{m\times n_2}$, $ f: \R^{n_1}\rightarrow (-\infty,\infty]$ and $g: \R^{n_2}\rightarrow (-\infty,\infty]$ are proper lower semi-continuous and convex functions. 
One of the most popular methods to solve the problem (\ref{eq-prob}) is alternating direction method of multiplier i.e., ADMM
\cite{Gabay-1976,Glowinski-1975}.
The convergence analysis of the traditional ADMM often assumes that the penalty parameter is fixed; see for example \cite{O1, O2, CST-2017, He-convergence, WY}. Because the efficiency of ADMM is highly sensitive to the penalty parameter, in practice, one would prefer to adaptively update the penalty parameter to avoid laborious tuning; see for example \cite{Liada,adab}. Existing works on the convergence of adaptive ADMM mostly assume that the ratio between two consecutive parameters tends to $1$ rapidly, and the algorithm quickly behaves just like the ADMM with a fixed penalty parameter \cite{He-adapt,xu2017adaptive,xu2017adaptive1,xu2017adaptive2}. In this paper, we aim to partially close the gap between theory and practice. In detail, we assume that one of the objective function $g(\cdot)$ is strongly convex, that is, for any $x\in \R^{n_2}$, $y\in \R^{n_2}$,
\begin{equation}\label{gstr}
 g(x)-g(y)\geq \< \xi,x-y\>+\sigma_g \ll x-y\ll^2/2,\ {\rm for\ any\ \xi\in\ }\partial g(y),
\end{equation}
where $\sigma_g>0$ is the strong convexity parameter and $\|\cdot\|$ stands for the Euclidean norm. We call this problem semi-strongly convex, which is also used in \cite{QTD}. With this assumption, we may greatly increase the freedom of adaptively adjusting the penalty parameter with guaranteed convergence. In Section~\ref{Sec-alg}, we will propose an adaptive ADMM with a special penalty updating scheme. That is, at every iteration, we define an interval to choose the new penalty parameter. The interval's length may tend to infinity, with floating lower bound and upper bound. We allow the parameter to increase to infinity at the rate of $O(k)$, where $k$ is the iteration counter, and decrease at a linear rate as long as there is a lower bound. We obtain various convergence results within this framework, which will be described in the next subsection.

\subsection{Convergence analysis}

In Section~\ref{Sec-conv}, we will analyse the convergence property of our algorithm. We first prove that our algorithm achieves accelerated convergence rate of $O(1/k^2)$ in terms of objective function value and primal feasibility. Accelerating algorithms for constrained optimization problems has been an active research area; see \cite{AC,AC1,HLI,UY,QTD,QTD1} for examples.  Since Goldstein et al. \cite{Goldstein-2014} proposed an accelerated ADMM with 
the convergence rate of $O(1/k^2)$ by making rather strong assumptions including one
that assumes both $f$ and $g$ are strongly convex, 
various attempts have been made to weaken the assumptions while maintaining the convergence rate of $O(1/k^2)$. 
In  \cite{XYY}, 
Xu  proposed an accelerated ADMM by increasing the penalty parameter while assuming that one of the component objective functions $g(\cdot)$ is strongly convex. This work significantly weakens the assumptions of Goldstein et al.
Since then, Xu's framework of increasing the penalty parameter has been generalised and modified by other researchers, see \cite{QTD2, Teboulle, XYY1}. Among them, Tran-Dinh increases the penalty parameter at a quadratic rate $O(k^2)$ to achieve the non-ergodic convergence rate of $O(1/k^2).$ Although the technique of increasing the penalty parameter can result in a nice convergence rate, this framework has two issues that prevent it from being practical. First, the convergence analysis focuses on the objective function value and primal feasibility, which doesn't involve the dual variables. In practice, we cannot check the optimality of a solution without the dual variable because the optimal objective value is unknown in advance. Therefore, it is necessary for us to analyse the convergence of the dual variable. Another issue is that, if we keep increasing the penalty parameter to very large values, the dual feasibility will not be penalised enough, and that will deteriorate the convergence speed
of the dual feasibility. This observation will also be illustrated in the numerical experiments. In order to overcome these two issues, we update the penalty parameter adaptively to balance the primal and dual KKT residue. Moreover, we prove  the convergence of primal-dual iterates, which has not been shown before for accelerated ADMM or adaptive ADMM. Our result implies that we can use the KKT residue  in the stopping criterion because the dual variable also converges. Apart from the sub-linear convergence rate, we also consider the condition for our algorithm to achieve linear convergence. For this aspect, our algorithm can achieve linear convergence if $g$ is  strongly convex and Lipschitz continuously differentiable, and the matrix $C$ has full row-rank. Note that the linear convergence of ADMM has been studied before, see \cite{DLinear, WY, HLinear, PG-linear, RN-linear} to just name a few. However, our analysis allows the penalty parameter to change adaptively. As far as we know, this paper is the first to analyse the linear convergence of an
adaptive ADMM.

\subsection{Partial proximal point method}
Because our adaptive method is designed to solve a semi-strongly convex problem, in Section~\ref{Sec-PPPM}, we consider how to apply it to solve problem (\ref{eq-prob}) if neither $f$ nor $g$ is strongly convex. The idea is that we add a proximal term to one of the variable and solve a sequence of problems. This method is called partial proximal point method (PPPM), which has been used in \cite{jiang2013solving} by Jiang et al. Note that our formulation is different from that in \cite{jiang2013solving} in the sense that their subproblem is strongly convex while our subproblem is only semi-strongly convex. 
While the convergence of the PPPM was done in \cite{PPPM}, the stopping conditions for solving the subproblems are based on practically unverifiable conditions. Here
we prove the convergence of the PPPM with the subproblems solved inexactly under verifiable conditions. Because the subproblem becomes a semi-strongly convex problem, we may use the adaptive ADMM to solve it. In Section~\ref{Sec-numer}, we  conduct numerical experiments to verify the efficiency of the PPPM against different types of ADMM.
\subsection{Organization of the paper}

In Section~\ref{Sec-alg}, we present our main algorithm. In Section~\ref{Sec-conv}, we 
conduct the convergence analysis of the algorithm. In Section~\ref{Sec-PPPM}, we discuss the partial proximal point method and its convergence analysis. In Section~\ref{Sec-imp}, we discuss some implementation strategies of our algorithm. In Section~\ref{Sec-numer}, we present numerical results to verify the robustness, convergence rate and efficiency of IADMM. In Section~\ref{Sec-conc}, we give a brief conclusion. The proofs of some results are put in the appendix.

\section{Self-adaptive ADMM}\label{Sec-alg}

\subsection{Preliminaries}
Before we state our algorithm, we provide several useful definitions and notations. Define $x=(y,z)$, we say that $(y,z,\lambda)$ is KKT
  solution of (\ref{eq-prob}) if the following conditions hold.
\begin{equation}\label{KKTmap} 
\begin{cases}0\in \partial f(y)+B^\top \lambda\\0\in \partial g(z)+C^\top \lambda\\By+Cz-b=0.\end{cases}
\end{equation}
Define
$$
\F(x):=f(y)+g(z), \;\; \A x:=By+Cz, \;\;\L(x,\lambda):=\F(x)+\inprod{\lambda}{\A x-b}.
$$ 
We assume that the KKT solution set for \eqref{eq-prob} is nonempty and let
 $(x^*:=(y^*,z^*), \lambda^*)$ be a KKT solution of (\ref{eq-prob}). Then we have $0\in \partial_x \L(x^*,\lambda^*).$ Hence $\L(x,\lambda^*)-\L(x^*,\lambda^*)\geq 0$, and $\L(x,\lambda^*)-\L(x^*,\lambda^*)=0$ if and only if $0\in \partial_x \L(x,\lambda^*)$. Thus, we have the following result
 
 \begin{equation}\label{Feb_16_1}
 \L(x,\lambda^*)-\L(x^*,\lambda^*)=0,\ \ll \A x-b\ll^2=0 \iff (x,\lambda^*)\ {\rm is\ a\ KKT\ solution.}
 \end{equation}
 For a given $n\times n$ symmetric positive semidefinite matrix $D$ and vectors $x,y,z\in \R^n$, define 
$$\eta_D(x,y,z):=\< D(z-y),x-z\>, \quad 
\xi_D(x,y,z):=\frac{1}{2}\ll x-y\ll^2_D-\frac{1}{2}\ll x-z\ll^2_D,
$$ 
where $\|w\|_D := \sqrt{\< w,Dw\>}$ for any $w\in \R^n$. Simple calculation shows that $\eta_{D}(x,y,z)=\xi_D(x,y,z)-\ll y-z\ll^2_D/2.$

\subsection{Algorithm statement}
Now we present our algorithm as follows.

\begin{algorithm}
	\renewcommand{\algorithmicrequire}{\textbf{Input:}}
	\renewcommand{\algorithmicensure}{\textbf{Output:}}
	\caption{IADMM}
	\label{alg:1}
	\begin{algorithmic}
		
		\STATE {\bf Initialization}: Choose $ (x^1,\lambda^1)$ and set constant parameters $\gamma\in \big(1,\frac{1+\sqrt{5}}{2}\big)$, $\epsilon, \beta,\tau\in (0,1)$, initial penalty parameter $\beta_1\in [\beta,+\infty).$ Choose a matrix $Q\succeq 0$. For all $k\geq 1$, 
		choose matrix $P_k\succeq 0\ {\rm such\ that}\ \beta_{k+1}P_{k+1}\preceq \beta_k P_k.$ Let $Q_k=\beta_k Q$.
		\FOR{$k=1,2,\dots$}		
		\STATE 1, $y^{k+1}\in \arg\min_y\left\{ f(y)+\< \lambda^k,By\>+\frac{\beta_k}{2}\ll By+Cz^k-b\ll^2+\frac{1}{2}\ll y-y^k\ll^2_{P_k}\right\}$
		\STATE 2, $z^{k+1}\in \arg\min_z\left\{ g(z)+\< \lambda^k,Cz\>+\frac{\beta_k}{2}\ll By^{k+1}+Cz-b\ll^2+\frac{1}{2}\ll z-z^k\ll^2_{Q_k}\right\}$
		\STATE 3, $\lambda^{k+1}=\lambda^{k}+\gamma\beta_k({By^{k+1}+Cz^{k+1}-b})$
		\STATE 4, Choose $\beta_{k+1}\in \left[\max\{\beta,\tau\beta_k\},\ \sqrt{\beta_k^2+\frac{(1-\epsilon)\sigma_g\beta_k}{\lambda_{\max}(C^\top C+Q)}}\right]$
		\ENDFOR
	\end{algorithmic}  
\end{algorithm}

Algorithm~\ref{alg:1} is similar to the traditional (proximal) ADMM. The only difference is that in step 4, we choose a new penalty parameter in an interval containing the current penalty parameter. This is why we call the algorithm IADMM, where "I" stands for interval.  The parameters $(\epsilon,\beta)>0$ are introduced only for theoretical analysis. In practice, we may choose $(\epsilon,\beta)$ to be small numbers. For simplicity, we only consider the case where
the step-length $\gamma> 1$ since in practice this choice typically will lead to a
faster convergence compared to the case where $\gamma\in (0,1)$. But note that our IADMM still works for the latter case. Some remarks on Algorithm IADMM are in order. 
First, the algorithm is still applicable to the case where the function $g(\cdot)$ is 
not strongly convex, i.e., $\sigma_g = 0$. In this case, the penalty parameters $\{\beta_k\}$
must be non-increasing.
Second, when the parameters $\beta_k$ is fixed for all
$k$, IADMM reduces to the proximal ADMM in \cite{O2} when we set $P_k = \beta_k P$ 
for some given $P\succeq 0$.
Third, we can also add a smooth function to the $g$-part, and perform a 
majorization in every iteration like the algorithm in \cite{XYY}. The convergence analysis is similar but includes more tedious details. For simplicity, we only consider problem (\ref{eq-prob}). Last, we assume that every subproblem is well-defined with an optimal solution.

\section{Convergence rate analysis}\label{Sec-conv}

In this section, we will analyse the convergence property of IADMM. We first state some useful lemmas. Their proofs are put in the appendix.

\subsection{Useful lemmas}
The following lemma serves as the foundation in the convergence analysis of our IADMM. Many theorems later are based on this lemma. Note that it holds even if $\sigma_g= 0$.

\begin{lem}\label{onestep}
Let $\delta := 1+\gam-\gam^2 >0$.
Then for any $(x,\lambda)$ satisfying $\A x-b=0$, we have
\begin{align}\label{onestepineq}
&\beta_k\left( \L(x^{k+1},\lambda)-\L(x,\lambda)\right) +
\frac{\delta\beta_{k-1}^2}{2\gamma}\ll \A x^{k}-b\ll^2
+ \frac{\delta\beta_k^2}{2}\norm{C(z^{k+1}-z^k)}^2
\notag
\\
& +\frac{\beta_k}{2}\norm{y^k-y^{k+1}}_{P_k}^2 
 +\frac{\beta_k^2}{2} \norm{z^k-z^{k+1}}_Q^2 
\notag \\
& \leq \Phi_k(x,\lambda) - \Phi_{k+1}(x,\lambda) 
-\sigma_g\beta_k\norm{ z^k-z^{k+1}}^2-\frac{\epsilon \sigma_g\beta_k}{2}\norm{ z-z^{k+1}}^2, 
\end{align}
where
\begin{eqnarray} \label{eq-Phi}
\Phi_k(x,\lambda) &=& \frac{1}{2\gamma} \norm{\lambda-\lambda^k}^2
+ \frac{(2-\gamma)\beta_{k-1}^2}{2}\norm{\A x^k-b}^2
+ \frac{\beta_k^2}{2}\norm{ z-z^k}^2_{C^\top C+Q}
\notag\\
&&
+\frac{\beta_{k-1}^2}{2}\norm{z^k-z^{k-1}}^2_Q
+ \frac{\beta_k}{2}\norm{y-y^k}_{P_k}^2.
\end{eqnarray}
\end{lem}

In the above lemma, $\Phi_k(x,\lambda)$ serves as a kind of energy function
for us to measure the progress of each IADMM iteration. In particular, the left-hand-side
of \eqref{onestepineq} gives the reduction in the ``energy'' one can expect at each iteration.

The next lemma also appears in Xu's convergence analysis of accelerated ADMM in \cite{XYY}.
\begin{lem}\label{useful}
Consider a continuous function $D(\lambda).$ Suppose for any $\lambda\in \R^m,$ $\L(x^k,\lambda)-\L(x^*,\lambda)\leq h(k)D(\lambda),$ where $h(k)\geq 0.$ Then $\ll \A x^k-b\ll=O(h(k))$, $| \F(x^k)-\F(x^*)|=O(h(k))$
\end{lem}

With Lemma~\ref{onestep} and Lemma~\ref{useful}, we are able to state and prove the convergence results about accelerated convergence, primal-dual iterative convergence and linear convergence in the following three subsections respectively.

\subsection{Ergodic convergence rate of $O(1/k^2)$}
\begin{theo}\label{ergodic}
Define $v^k:=\frac{\sum_{i=1}^{k}\beta_i x^{i+1}}{\sum_{i=1}^k\beta_i}$ and $\gamma_k:=\sum_{i=1}^k \beta_i$ Then

$$|\F(v^k)-\F(x^*)|=O(1/\gamma_k),\ \ll \A v^k-b\ll=O(1/\gamma_k).$$
\end{theo}

\begin{proof}
Since the right-hand-side of (\ref{onestepineq}) in Lemma \ref{onestep}
 is summable, we choose $x=x^*$ for some optimal solution and take summation of the inequality (\ref{onestepineq}) from $1$ to $k$, then we get
\begin{equation}\label{Feb_16_2}
\sum_{i=1}^{k}\beta_i \left( \L(x^{i+1},\lambda)-\L(x^*,\lambda)\right)\leq 
D(\lambda) := \Phi_1(x^*,\lambda).
\end{equation}
Note that when deriving (\ref{Feb_16_2}), we have ignored many nonnegative terms in (\ref{onestepineq}). From the convexity of $\L(x,\lambda)$ as a function of $x$, we have 
\begin{equation}
\gamma_k\left( \L(v^k,\lambda)-\L(x^*,\lambda)\right)\leq D(\lambda).
\end{equation}
By applying Lemma~\ref{useful} to the above inequality, we get Theorem~\ref{ergodic}.
\end{proof}

Since $\beta_i\geq \beta$ for any $i$, then $\gamma_k=\Omega(k)\footnote{A sequence $\{a_k\}_{k\in \mathbb{N}^+}$ is said to be $\Omega(k)$ if there exists some positive number $c$ and integer $N_0$ such that $a_k\geq c k$ for any $k\geq N_0.$}.$ The following corollary can be derived from Theorem~\ref{ergodic} directly.

\begin{coro}\label{ergodic-convergence}
Let $v^k$, $\gamma_k$ be defined as in Theorem~\ref{ergodic}. Then we have $$|\F(v^k)-\F(x^*)|=O(1/k),\ \ll \A v^k-b\ll=O(1/k).$$ Moreover, if $\gamma_k=\Omega(k^2),$ then $$|\F(v^k)-\F(x^*)|=O(1/k^2),\ \ll \A v^k-b\ll=O(1/k^2).$$
\end{coro}

\noindent{\bf Remark.} Note that from Corollary~\ref{ergodic-convergence}, the convergence rate is at least $O(1/k)$, even if $\sigma_g=0$. Also, when $\sigma_g > 0$, we see that it is possible to choose $\beta_i$ such that $\beta_i=\Omega(i)$. Indeed, if we choose $\beta_{i+1}$ to be the upper bound of the interval in Step 4 of Algorithm 1 at every iteration, we get $\beta_i=\Omega(i)$ and so $\gamma_k=\Omega(k^2).$ Thus, our algorithm can achieve the convergence rate 
of $O(1/k^2)$ when $\sigma_g > 0.$ However, we should note that  even though the objective value gap and primal feasibility decrease at the rate of $O(1/k^2)$, the dual feasibility may not. In the next subsection, we will establish the convergence the the sequence $(x^k,\lambda^k ).$


\subsection{Nonergodic convergence of iteration points}
In this section, we give a proof of the convergence of the iteration points of IADMM. Suppose $(x^*,\lambda^*)$ is a KKT solution. Our convergence theorem is as follows.

\begin{theo}\label{iteration}
Suppose $\beta_kP_k+\beta_k^2 B^\top B\succeq \Theta \;\forall\; k$ for some $\Theta\succ 0$. Then $(y^k,z^k,\lambda^k)$ converges to a KKT solution $(y^*,z^*,\lambda^*)$ as $k\to\infty.$
\end{theo}

\begin{proof}
If we choose $(x,\lambda)=(x^*,\lambda^*)$ in Lemma~\ref{onestep} and define
the quantity
\begin{eqnarray}\label{defiphi}
\Phi_k :=  \Phi_k(x^*,\lambda^*) &=&  \frac{1}{2\gamma} \norm{\lambda^*-\lambda^k}^2 +
\frac{(2-\gamma)\beta_{k-1}^2}{2}\ll \A x^k-b\ll^2+\frac{\beta_k^2}{2}\ll z^*-z^k\ll_{C^\top C+Q}^2\notag
\\
&& +\frac{\beta_{k-1}^2}{2}\ll z^k-z^{k-1}\ll^2_{Q}
+\frac{\beta_k}{2} \norm{ y^*-y^k}^2_{P_k},
\end{eqnarray}
we get
\begin{eqnarray}\label{iter}
\left. \begin{array}{l}
\frac{\delta\beta_{k-1}^2}{2\gamma}\ll \A x^{k}-b\ll^2
+\frac{\beta_k^2}{2}\norm{ z^k-z^{k+1}}^2_Q
+\frac{\delta\beta_k^2}{2}\norm{C(z^k-z^{k+1})}^2
\\[5pt]
+\frac{\beta_k}{2}\norm{ y^k-y^{k+1}}^2_{P_k}
+\sigma_g\beta_k\ll z^k-z^{k+1}\ll^2+\frac{\epsilon \sigma_g \beta_k}{2}\ll z^*-z^{k+1}\ll^2
\end{array}\right\} 
\leq \Phi_k-\Phi_{k+1},
\end{eqnarray}
Note that when deriving the above inequality, we have used the inequality $\L(x^{k+1},\lambda^*)-\L(x^*,\lambda^*)\geq 0.$ By taking summation in \eqref{iter}, we have that 
the infinite sum on the left-hand-side sequence is finite. Thus we have the following fact.
\\[5pt]
\noindent{\bf Fact 1.} 
$ \beta_{k-1}^2 \ll \A x^{k}-b\ll^2=o(1)$, 
$\beta_k\norm{ y^k-y^{k+1}}^2_{P_k}=o(1),$  
$\beta_k^2\norm{ z^k-z^{k+1}}^2_{C^\top C+Q}=o(1),$
$\beta_k\norm{z^{k+1}-z^*}^2=o(1)$, 
and 
\begin{equation}\label{Feb_16_3}
\sum\limits_{k=1}^\infty \frac{\sigma_g}{\beta_k}\left( \beta_k^2\ll z^k-z^{k+1}\ll^2+\frac{\epsilon \beta_k^2}{2}\ll z^*-z^{k+1}\ll^2\right) \;<\; \infty.
\end{equation}
Moreover, for the sequence of parameters $\{\beta_{k}\}$ in Step 4, 
 we can easily prove that $\beta_k\leq \beta_1+\frac{(1-\epsilon)\sigma_g (k-1)}{\lambda_{\max}(C^\top C+Q)}=O(k),$ which implies that $\sum_{k=1}^{\infty}1/\beta_k=\infty.$
Together with the fact that $\sigma_g > 0$ and (\ref{Feb_16_3}), we have the following result.
\\[5pt]
\noindent{\bf Fact 2.}  $\liminf\limits_{k\rightarrow \infty}\beta_k^2 \ll z^k-z^{k+1}\ll^2+\frac{\epsilon\beta_k^2}{2}\ll z^*-z^{k+1}\ll^2=0.$

\medskip
From \eqref{iter}, we know that $\{\Phi_k\}$ is a nonincreasing sequence and it is bounded. Hence, from the definition of $\Phi_k$ in (\ref{defiphi}), we have the boundedness of the following sequences:
\begin{eqnarray}
 \{ \norm{\lambda^*-\lambda^k}\}, \quad \{\sigma_g\beta_k\norm{z^*-z^k}^2\},
  \quad \{\beta_k^2\norm{z^*-z^k}^2_{C^\top C + Q}\}, \quad
 \{\beta_k\norm{y^*-y^k}_{P_k}^2 \},
 \label{eq-bounded}
\end{eqnarray} 
where the boundedness of the second term comes from Fact 1 and the boundedness of $\{ \beta_{k+1}/\beta_k\}.$
Since $\{\beta_{k}/\beta_{k+1}\}$ is bounded, from the third sequence in (\ref{eq-bounded}), $\{ \beta_{k}^2 \norm{C(z^{k+1}-z^*)}^2\}$ is also bounded.
Next we show that $\{\norm{y^*-y^k}\}$ is bounded. 
From the convexity of the function $\|\cdot\|^2$ and $\A x^{k+1}-b=B(y^{k+1}-y^*)+C(z^{k+1}-z^*),$ we have the following inequality
\begin{eqnarray*}
\beta_k^2 \norm{B(y^{k+1}-y^*)}^2 
\leq 2 \beta_k^2 \norm{\A x^{k+1} -b}^2 + 2 \beta_k^2 \norm{C(z^{k+1}-z^*)}^2,
\end{eqnarray*}
then the boundedness of $\{\beta_k^2 \norm{\A x^{k+1} -b}^2 \}$ in Fact 1 and that of
$\{ \beta_k^2 \norm{C(z^{k+1}-z^*)}\}$ (just mentioned above) imply that 
$\{ \beta_k^2 \norm{y^{k+1}-y^*}^2_{B^\top B}\}$ is  bounded. 
Because $\{\beta_{k+1}/\beta_k\}$ is bounded,  $\{ \beta_{k+1}^2 \norm{y^{k+1}-y^*}^2_{B^\top B}\}$ is also bounded. From the condition in the theorem, we have the fact that 
$\norm{y^*-y^k}_\Theta^2 \leq \beta_k \norm{y^*-y^k}_{P_k}^2 + 
 \beta_k^2 \norm{y^*-y^k}_{B^\top B}^2$ and $\Theta \succ 0$. This implies that 
$\{\norm{y^*-y^k}\}$ is bounded. 

From the above results, we can conclude that $\{(y^k,z^k,\lambda^k)\}$ is bounded. From Fact 1, we have that $z^k\rightarrow z^*.$ Combine these two results and Fact 2, we can see that there exists a sequence $\{ j_k\}_{k\geq 1}$ 
such that $(y^{j_k+1},z^{j_k+1},\lambda^{j_k+1})$ converges to a limit point 
$(y,z^*,\lambda)$, and 
$\beta_{j_k}^2\ll z^{j_k}-z^{j_k+1}\ll^2=o(1),$ 
$\beta_{j_k}^2\ll z^*-z^{j_k+1}\ll^2=o(1)$.
We summarize the results as follows.
\\[5pt]
\noindent{\bf Fact 3.} $(y^{j_k+1},z^{j_k+1},\lambda^{j_k+1})\rightarrow (y,z^*,\lambda)$, $\beta_{j_k}^2\ll z^{j_k}-z^{j_k+1}\ll^2=o(1),$ $\beta_{j_k}^2\ll z^*-z^{j_k+1}\ll^2=o(1).$

\medskip
We will next show that $\norm{y^{j_k}-y^{j_k+1}} = o(1).$ 
First, from the fact that $\beta_k^2\norm{\A x^{k+1}-b}^2 = o(1)$ in Fact 1 and the 
boundedness of $\{\beta_{k+1}/\beta_k\}$, we can readily show
that $\beta_k^2\norm{\A(x^{k+1}-x^k)}^2 = o(1).$ Next, from 
\begin{eqnarray*}
 & & \hspace{-7mm}\beta_{j_k}^2\norm{B(y^{j_k+1}-y^{j_k})}^2 
 \leq 2 \beta_{j_k}^2 \norm{\A(x^{j_k+1}-x^{j_k})}^2 
 + 2\beta_{j_k}^2 \norm{C(z^{j_k+1}-z^{j_k})}^2
 \\[5pt]
 &\leq & 2 \beta_{j_k}^2 \norm{\A(x^{j_k+1}-x^{j_k})}^2 
 + 2\lambda_{\max}(C^\top C)\beta_{j_k}^2 \norm{z^{j_k+1}-z^{j_k}}^2,
\end{eqnarray*}
and Fact 3, we get $\beta_{j_k}^2\norm{B(y^{j_k+1}-y^{j_k})}^2 = o(1).$
From Fact 1, we also have 
$\beta_{j_k}\norm{y^{j_k}-y^{j_k+1}}_{P_{j_k}}^2 = o(1)$.
Thus, from the condition in the theorem, $
\norm{y^{j_k+1}-y^{j_k}}^2_\Theta 
\leq  \beta_{j_k}\norm{y^{j_k}-y^{j_k+1}}_{P_{j_k}}^2 + 
\beta_{j_k}^2\norm{B(y^{j_k+1}-y^{j_k})}^2 = o(1).$
Since $\Theta \succ 0$, this implies that $\norm{y^{j_k+1}-y^{j_k}} = o(1).$

\medskip
From \eqref{eq-opt-f} and \eqref{eq-opt-g} in the Appendix, which are the optimality conditions
in Step 1 and Step 2 of Algorithm 1, we know that 
\begin{align}\label{demi}
&P_{j_k}(y^{j_k}-y^{j_k+1})+\beta_{j_k}B^\top C(z^{j_k+1}-z^{j_k})+(\gamma-1)\beta_{j_k} B^\top (\A x^{j_k+1}-b)\in \partial f(y^{j_k+1})+B^\top \lambda^{j_k+1}\notag \\
&Q_{j_k}(z^{j_k}-z^{j_k+1})+(\gamma-1)\beta_{j_k}C^\top (\A x^{j_k+1}-b)\in \partial g(z^{j_k+1})+C^\top \lambda^{j_k+1}.
\end{align}
Since $\norm{y^{j_k+1}-y^{j_k}} = o(1)$ and $0\preceq P_{j_k} \preceq (\beta_{j_1}/\beta) P_{j_1}$, we have that $P_{j_k}(y^{j_k}-y^{j_k+1})\rightarrow 0$. From Fact 1, all the following terms, $\beta_{j_k}B^\top C(z^{j_k+1}-z^{j_k})$, $Q_{j_k}(z^{j_k}-z^{j_k+1})$, $\beta_{j_k}C^\top (\A x^{j_k+1}-b)$ and $\beta_{j_k}B^\top(\A x^{j_k+1}-b)$ converge to 0. With all of the mentioned convergence results and the demi-closedness of $\partial f$ and $\partial g$, after letting $k\rightarrow \infty$ in (\ref{demi}), we have that
\begin{equation}
0\in \partial f(y)+B^\top \lambda, \quad \ 0\in \partial g(z^*)+C^\top \lambda.\notag
\end{equation} 
Together with $\A x-b=\lim_{k\rightarrow \infty}\A x^{j_k+1}-b=0$, we know that $(y,z^*,\lambda)$ is a KKT solution. For convenience, we let $(y,z^*,\lambda)=(y^*,z^*,\lambda^*).$

From \eqref{iter}, we know that
$0\leq \Phi_{k+1} \leq \Phi_k$ and $\lim_{k\to\infty} \Phi_k$ exists.  From the condition of $P_k$ in Algorithm~\ref{alg:1}, we have that $\beta_{j_k+1}\| y^*-y^{j_k+1}\|^2_{P_{j_k+1}}\leq \beta_1\lambda_{\max}(P_1)\| y^*-y^{j_k+1}\|^2=o(1)$.
Combine this and Fact 1, Fact 3, the boundedness of $\{\beta_{j_k+1}/\beta_{j_k}\}$ and the definition of $\Phi_{k}$, we have that
$\lim_{k\to \infty}\Phi_{j_{k}+1} = 0$. Thus
$\lim_{k\to \infty} \Phi_k = 0$. From here, we get
\begin{eqnarray}\label{Feb_16_4}
 \beta_k\norm{y^*-y^k}_{P_k}^2 = o(1), \quad
 \beta_k^2\norm{z^*-z^k}^2_{C^\top C + Q} = o(1), \quad \| \lambda^*-\lambda^k\|^2=o(1),
\end{eqnarray}
which implies that $\lim_{k\rightarrow \infty}\lambda^k=\lambda^*.$ Note that from Fact 1 and $\beta_k\geq \beta>0$ for all $k$,
we have $\norm{z^*-z^k} = o(1)$, i.e., $\lim_{k\to\infty}z^k = z^*$. 
Moreover, from $\lim_{k\rightarrow \infty}\Phi_k=0,$ we have that
\begin{eqnarray}\label{Feb_16_5}
 \beta_k^2\norm{B(y^k-y^*)}^2 \leq 2 \Big(\frac{\beta_k}{\beta_{k-1}}\Big)^2\beta_{k-1}^2
 \norm{\A x^k - b}^2 + 2\beta_k^2\norm{C(z^k-z^*)}^2 = o(1).
\end{eqnarray}
From (\ref{Feb_16_4}), (\ref{Feb_16_5}), $\norm{y^*-y^k}_\Theta^2 \leq \beta_k\norm{y^*-y^k}_{P_k}^2 + 
\beta_k^2 \norm{B(y^k-y^*)}^2=o(1)$ and $\Theta \succ 0$, we get
$\norm{y^*-y^k}=o(1)$, i.e., $\lim_{k\to\infty}y^k = y^*$.
The proof is completed.
\end{proof}

\begin{rem} In Theorem \ref{iteration}, we have assumed that $\sigma_g >0$. 
By modifying the proof slightly, one can prove that the theorem also holds
when $\sigma_g = 0$, provided $C^\top C + Q \succ 0$. Note that in this case, $\beta_k$ decreases monotonically.
More specifically, when $C^\top C + Q \succ 0$, one can see from 
\eqref{iter} that $\{\beta_k^2\norm{z^*-z^k}^2\}$ is bounded.
From there, one can show that the results in Fact 3 are valid, and the 
rest of the proof of  Theorem \ref{iteration} can carry through.
\end{rem}

\subsection{Nonergodic linear convergence}\label{Subsec-linear}
In this subsection, in addition to assuming that $g$ is strongly convex with parameter $\sigma_g >0$, we also
assume that  it is continuously differentiable and
$\nabla g$ is Lipschitz continuous with parameter $L_g$.
Moreover, we choose  $P_k=0$ for all $k\geq 1$
and assume that Step 1 of Algorithm 1 is well defined. To begin with, we choose $(x,\lambda)=(x^*,\lambda^*)$ in Lemma~\ref{onestep} to get the following lemma.

\begin{lem}\label{energy1} For any KKT solution $(x^*,\lambda^*),$
\begin{align}
&\beta_k\left( \L(x^{k+1},\lambda^*)-\L(x^*,\lambda^*)\right)+
\frac{(2-\gamma)\beta_k^2}{2}\norm{\A x^{k+1}-b}^2
+\left(\beta_k^2+\frac{\sigma_g\beta_k}{\lambda_{\max}(Q)}\right)\ll z^{k+1}-z^k\ll^2_Q\notag\\
&+\left(\frac{\beta_{k+1}^2}{2}+\frac{\epsilon \sigma_g\beta_k}{4\lambda_{\max}(C^\top C+Q)}\right)\ll z^*-z^{k+1}\ll^2_{C^\top C+Q}+\frac{1}{2\gamma}\ll \lambda^*-\lambda^{k+1}\ll^2\notag \\
&\leq \frac{(2-\gamma-\delta/\gamma)\beta_{k-1}^2}{2}\ll \A x^k-b\ll^2
+\frac{\beta_{k-1}^2}{2}\ll z^k-z^{k-1}\ll^2_Q+\frac{\beta_k^2}{2}\ll z^*-z^k\ll^2_{C^\top C+Q}+\frac{1}{2\gamma}\ll \lambda^*-\lambda^{k}\ll^2\notag \\
&\quad -\frac{\epsilon \sigma_g\beta_k}{4}\ll z^*-z^{k+1}\ll^2.\label{ieqlong}
\end{align}
\end{lem} 

We also need the following lemma, which is motivated from Lemma 3.2 in \cite{WY}.

\begin{lem}\label{wy} For any $0<\alpha<\frac{1}{2}$,
\begin{align}\label{thir6}
&(1-2\alpha)\lambda_{\min}(CC^\top) \ll \lambda^{k+1}-\lambda^*\ll^2\notag \\
&\leq 
\frac{1}{\alpha}\lambda_{\max}(CC^\top)(1-\gamma)^2\beta_k^2\ll \A x^{k+1}-b\ll^2+\frac{1}{\alpha}\lambda_{\max}(Q)\beta_k^2\ll z^{k+1}-z^k\ll^2_Q+L_g^2\ll z^{k+1}-z^*\ll^2.
\end{align}
\end{lem}

Now, we are ready to state the main convergence theorem.

\begin{theo}\label{linearc}
Suppose $\nabla g$ is Lipschitz continuous with parameter $L_g$ and $C$ has full row rank. Choose $P_k=0$ for all $k$ and assume that Step 1 of Algorithm 1 is well-defined. Suppose $\beta\leq \beta_k\leq \bar{\beta}$ for any $k$, then $\L(x^k,\lambda^*)-\L(x^*,\lambda^*)$, $\ll \A x^{k}-b\ll^2$ and $\ll \lambda^k-\lambda^*\ll^2$ converges to zero R-linearly as $k\to \infty.$
\end{theo}

\begin{proof}
Multiply the inequality in 
Lemma~\ref{wy} by $\phi>0$ such that $\phi\leq \epsilon \sigma_g\beta/(4 L_g^2)$ 
and add it to the inequality in Lemma~\ref{energy1}, we get
\begin{align}
&\beta_k\left( \L(x^{k+1},\lambda^*)-\L(x^*,\lambda^*)\right)
+\left( 2-\gamma- \frac{2\phi}{\alpha}\lambda_{\max}(CC^\top)(1-\gamma)^2\right)\frac{\beta_k^2}{2}\norm{\A x^{k+1}-b}^2\notag \\
&+\left( \beta_k^2+\frac{\sigma_g\beta_k}{\lambda_{\max}(Q)}
-\frac{\phi}{\alpha}\lambda_{\max}(Q)\beta_k^2\right)\ll z^{k+1}-z^k\ll^2_Q\notag \\
&+\left( \frac{\beta_{k+1}^2}{2}+\frac{\epsilon \sigma_g\beta_k}{4\lambda_{\max}(C^\top C+Q)}\right)\ll z^*-z^{k+1}\ll^2_{C^\top C+Q}+\left( \frac{1}{2\gamma}+\phi (1-2\alpha)\lambda_{\min}(CC^\top)\right)\ll \lambda^*-\lambda^{k+1}\ll^2\notag \\
&\leq (2-\gamma-\delta/\gamma)\frac{\beta_{k-1}^2}{2}\ll \A x^k-b\ll^2+\frac{\beta_{k-1}^2}{2}\ll z^k-z^{k-1}\ll^2_Q+\frac{\beta_k^2}{2}\ll z^*-z^k\ll^2_{C^\top C+Q}
+\frac{1}{2\gamma}\ll \lambda^*-\lambda^{k}\ll^2.\notag 
\end{align}
Note that in the last inequality, we removed the nonpositive term
$(\phi L_g^2 -\frac{\epsilon\sigma_g\beta}{4}) \ll z^{k+1}-z^*\ll^2$.
By using the fact that $\beta\leq \beta_k\leq \bar{\beta}$ for any $k\geq 0$ in the above inequality, we get
\begin{align}
&\beta_k\left( \L(x^{k+1},\lambda^*)-\L(x^*,\lambda^*)\right)
+\left( 1 + 
\frac{\delta/\gamma - 2\phi \lambda_{\max}(CC^\top)(\gamma-1)^2/\alpha}{2-\gamma-\delta/\gamma}\right)
\frac{(2-\gamma-\delta/\gamma)\beta_k^2}{2}\ll \A x^{k+1}-b\ll^2\notag \\
&+\left( 2+\frac{2\sigma_g}{\lambda_{\max}(Q)\bar{\beta}}
-\frac{2\phi}{\alpha}\lambda_{\max}(Q)\right)\frac{\beta_k^2}{2}\ll z^{k+1}-z^k\ll^2_Q
\notag \\
&
+\left(1+\frac{\epsilon \sigma_g\beta}{2\lambda_{\max}(C^\top C+Q)\bar{\beta}^2}\right)\frac{\beta_{k+1}^2}{2}\ll z^*-z^{k+1}\ll^2_{C^\top C+Q}
+\left(1+ 2\gamma\phi  (1-2\alpha)\lambda_{\min}(CC^\top)\right) \frac{1}{2\gamma}\ll \lambda^*-\lambda^{k+1}\ll^2\notag \\
&\leq
\E(k) :=  \frac{(2-\gamma-\delta/\gamma)\beta_{k-1}^2}{2}\ll \A x^k-b\ll^2+\frac{\beta_{k-1}^2}{2}\ll z^k-z^{k-1}\ll^2_Q+\frac{\beta_k^2}{2}\ll z^*-z^k\ll^2_{C^\top C+Q}
+\frac{1}{2\gamma}\ll \lambda^*-\lambda^{k}\ll^2.
\notag 
\end{align}

Note that from $\delta=1+\gamma-\gamma^2$ and $\gamma>1$, we have $2-\gamma-\delta/\gamma>0.$ If we choose $\phi>0$ to be sufficiently small so that all coefficients in
the parentheses on the left-hand-side are positive, then we have that
\begin{equation}
  \E(k+1)\leq M(\phi)\E(k),
\end{equation}
where 
\begin{align}
&M(\phi):=\max\Bigg\{ \frac{1}{1 +\frac{\delta/\gamma-2\phi \lambda_{\max}(CC^\top)(\gamma-1)^2/\alpha}{2-\gamma-\delta/\gamma}}\,,\ 
\frac{1}{2+\frac{2\sigma_g}{\lambda_{\max}(Q)\bar{\beta}}-\frac{2\phi}{\alpha}\lambda_{\max}(Q)},\ \notag \\
&\hspace{2.6cm} \frac{1}{1+\frac{\epsilon\sigma_g\beta}{2\lambda_{\max}(C^\top C+Q)\bar{\beta}^2}},\ \frac{1}{1+2\gamma\phi(1-2\alpha)\lambda_{\min}(CC^\top)}\Bigg\}.\notag
\end{align}
Note that we have ignored the nonnegative term $\beta_k\left( \L(x^{k+1},\lambda^*)-\L(x^*,\lambda^*)\right).$ It is easy to see that if $\phi>0$ is sufficiently small, then $0<M(\phi)<1$, which implies that $\E(k)\rightarrow 0$ linearly. From the definition of $\E(k)$ and the lower boundedness of $\beta_k$, we can see that $\L(x^k,\lambda^*)-\L(x^*,\lambda^*)$, $\ll \A x^k-b\ll^2$ and $\ll \lambda^k-\lambda^*\ll^2$ all converge R-linearly to zero.

\end{proof}

\section{A partial proximal point method with IADMM for solving non-semi-strongly 
convex problem}\label{Sec-PPPM}

The theoretical analysis in the previous section is based on the semi-strongly convexity of the problem (\ref{eq-prob}). However, in practice, many composite programming problems may not
be semi-strongly convex. In this section, in order to resolve this issue, we introduce
a partial proximal point method to solve \eqref{eq-prob} where in each iteration of the 
algorithm, we solve 
the following perturbed subproblem with a  proximal term inexactly by our IADMM:
\begin{equation}\label{prox-prob}
(y_{k+1},z_{k+1}) \approx \min\left\{ f(y)+g(z)+\frac{\sigma}{2}\| z-z_k\|^2:\ By+Cz=b \right\},
\end{equation}
where $\sigma>0.$ The subproblem (\ref{prox-prob}) is similar to the subproblem of a proximal point method. However, we only add the proximal term to one of the variables. Since the subproblem (\ref{prox-prob}) is not equivalent to problem \eqref{eq-prob}, we will solve a sequence of such subproblems inexactly to obtain a solution of \eqref{eq-prob}.

\begin{algorithm}
	\renewcommand{\algorithmicrequire}{\textbf{Input:}}
	\renewcommand{\algorithmicensure}{\textbf{Output:}}
	\caption{PPPM}
	\label{alg:PPPM}
	\begin{algorithmic}
		\STATE {\bf Initialization}: Choose $(y_1,z_1)\in \R^{n_1}\times \R^{n_2}.$
		\FOR{$k=1,2,\dots$}	
		\STATE 1, Choose parameter $\sigma_k\geq 0$	
		\STATE 2, Compute an approximate KKT solution $(y_{k+1},z_{k+1},\lambda_{k+1})$ of (\ref{prox-prob}) with $\sigma=\sigma_k$ by IADMM.
		\ENDFOR
	\end{algorithmic}  
\end{algorithm}

 We use the following inexact KKT condition to measure the accuracy of the subproblem. 
\begin{equation}\label{iKKTmap} 
\begin{cases}
\epsilon_k^1\in \partial f(y_k)+B^\top \lambda_k \\
\epsilon_k^2\in \partial g(z_k)+C^\top \lambda_k+\sigma_k ( z_k-z_{k-1} )\\\epsilon_k^3=By_k+Cz_k-b
\end{cases}
\end{equation}
where $\(y_k,z_k,\lambda_k\)$ is the outcome of the $k$th subproblem and $\( \epsilon_k^1,\epsilon_k^2,\epsilon_k^3 \)\in \R^{n_1}\times \R^{n_2}\times \R^m$ is the error term. 
Note that we use subscripts to differentiate the iterations in the partial proximal point method and IADMM. 

The convergence analysis of inexact proximal point methods have been a popular research topic because of its wide applications and connection to augmented Lagrangian methods \cite{yang2022bregman,liang2021inexact,cui2019r,rockafellar1976augmented,eckstein2013practical}.
Here we give the convergence theorem of the PPPM to make the paper self-contained.

\begin{theo}\label{conv2}
Suppose $0\leq \sigma_{k-1}\leq \sigma_k,$ $\sum_{k=1}^\infty \( 1+\|y_k\| \)\|\epsilon_k^1\|<\infty,$ $\sum_{k=1}^\infty \( 1+\|z_k\| \)\|\epsilon_k^2\|<\infty$ and $\sum_{k=1}^\infty \( 1+\|\lambda_k\| \)\|\epsilon_k^3\|<\infty,$ then the sequence $(y_k,z_k,\lambda_k)$ generated from Algorithm~\ref{alg:PPPM} satisfies

$$\lim\limits_{k\rightarrow \infty}\max\left\{ \dist(0,\partial f(y_k)+B^\top \lambda_k), \dist(0,\partial g(z_k)+C^\top \lambda_k), \| \A x^k-b \|, |\F(x_k)-\F(x^*)| \right\}=0,$$
where $x^*$ is an optimal solution of (\ref{eq-prob}).
\end{theo}

\begin{proof}
Let $(x^*,\lambda^*)$ be a KKT solution of problem \eqref{eq-prob}. Combine (\ref{iKKTmap}) and the convexity of $f$ and $g$, we obtain the following inequalities
\begin{equation}\label{Feb_2_1}
f(y_k)+\< \epsilon_k^1-B^\top \lambda_k,y^*-y_k \>\leq f(y^*),
\end{equation}
\begin{equation}\label{Feb_2_2}
g(z_k)+\< \epsilon_k^2-C^\top \lambda_k-\sigma_k(z_k-z_{k-1}) ,z^*-z_k \>\leq g(z^*).
\end{equation}
Adding (\ref{Feb_2_1}) and (\ref{Feb_2_2}), we have that
\begin{eqnarray}
&& \F(x_k)+\< \epsilon_k^1,y^*-y_k \>+\< \epsilon_k^2,z^*-z_k \>+\< \lambda_k, B(y_k-y^*)+C(z_k-z^*)\> \nonumber \\
&& \leq \F(x^*)+\sigma_k \eta(z^*,z_{k-1},z_k). \label{Feb_2_3}
\end{eqnarray}
Using $By^*+Cz^*=b$ in (\ref{Feb_2_3}) together with Cauchy-Schwarz inequality, we get 
\begin{eqnarray}
\F(x_k)-\F(x^*)+\sigma_k\| z_{k-1}-z_k \|^2/2
&\leq & (\|y^*\|+\|y_k\|)\|\epsilon_k^1\|+ (\|z^*\|+\|z_k\|)\|\epsilon_k^2\|+\|\lambda_k\|\|\epsilon_k^3\|
\nonumber  \\ 
&& +\sigma_k\| z_{k-1}-z^*\|^2/2-\sigma_k\| z_{k}-z^*\|^2/2.
\label{Feb_2_4}
\end{eqnarray}
Moreover, from the convexity of $\L(x,\lambda^*)$ and $0\in \partial_x \L(x^*,\lambda^*)$ we have that $\L(x_k,\lambda^*)\geq \L(x^*,\lambda^*).$ This implies
\begin{equation}\label{Feb_2_5}
\F(x_k)\geq \F(x^*)-\<\lambda^*,\A x_k-b\>\geq \F(x^*)-\|\lambda^*\|\|\A x_k-b\|\geq \F(x^*)-\|\lambda^*\|\|\epsilon_k^3\|.
\end{equation}
Substitute (\ref{Feb_2_5}) into (\ref{Feb_2_4}) and $\sigma_{k-1}\geq \sigma_k$, we get 
\begin{multline}\label{Feb_2_6}
0\leq \F(x_k)-\F(x^*)+\sigma_k\| z_{k-1}-z_k \|^2/2+\|\lambda^*\|\|\epsilon_k^3\|\leq (\|y^*\|+\|y_k\|)\|\epsilon_k^1\|\\
+ (\|z^*\|+\|z_k\|)\|\epsilon_k^2\|+(\|\lambda^*\|+\|\lambda_k\|)\|\epsilon_k^3\|+\sigma_{k-1}\| z_{k-1}-z^*\|^2/2-\sigma_k\| z_{k}-z^*\|^2/2.
\end{multline}
From the condition in the theorem, we have that the sum of right-hand-side of (\ref{Feb_2_6}) is upper bounded. Thus, we get 
\begin{equation}\label{Feb_2_7}
\lim\limits_{k\rightarrow \infty}\F(x_k)-\F(x^*)+\sigma_k\| z_{k-1}-z_k \|^2/2+\|\lambda^*\|\|\epsilon_k^3\|=0.
\end{equation}
Using (\ref{Feb_2_5}) in (\ref{Feb_2_7}), we get 
\begin{equation}\label{Feb_2_8}
\lim\limits_{k\rightarrow \infty}\F(x_k)-\F(x^*)+\|\lambda^*\|\|\epsilon_k^3\|=0,\ \lim\limits_{k\rightarrow \infty}\sigma_k\| z_{k-1}-z_k \|^2=0.
\end{equation}
Because $0\leq \sigma_k\leq \sigma_1,$ we have that $\lim\limits_{k\rightarrow \infty}\sigma_k^2\| z_{k-1}-z_k \|^2=0.$ 
This implies that $\lim\limits_{k\rightarrow \infty}\sigma_k\| z_{k-1}-z_k \|=0.$ Substitute this into (\ref{iKKTmap}), we have that
\begin{equation}\label{Feb_2_9}
\lim\limits_{k\rightarrow \infty}\max\left\{ \dist(0,\partial f(y_k)+B^\top \lambda_k), \dist(0,\partial g(z_k)+C^\top \lambda_k), \| \A x^k-b \| \right\}=0.  
\end{equation}
Because $\lim\limits_{k\rightarrow \infty}\|\lambda^*\|\|\epsilon_k^3\|=0,$ from (\ref{Feb_2_8}), we get 
\begin{equation}\label{Feb_2_10}
\lim\limits_{k\rightarrow \infty}\F(x_k)-\F(x^*)=0.
\end{equation}
Combining (\ref{Feb_2_9}) and (\ref{Feb_2_10}), we get Theorem~\ref{conv2}. 
\end{proof}

Because from Theorem~\ref{iteration}, the sequence generated of IADMM is convergent. We can control the size of $(1+\|y_k\|)\|\epsilon_k^1\|,$ $(1+\|z_k\|)\|\epsilon_k^2\|$ and $(1+\|\lambda_k\|)\|\epsilon_k^3\|$ by solving the subproblem accurately enough. Therefore, the condition in Theorem~\ref{conv2} can be guaranteed.

\section{Practical implementation}\label{Sec-imp}
In this section, we will move on to consider the practical usage of Algorithm~\ref{alg:1} and Algorithm~\ref{alg:PPPM}.

\subsection{Application in LASSO type problems}\label{Subsec-LASSO}
In this subsection, we consider the problem 
\begin{equation}\label{unconstraint}
\min \left\{ f(\lambda)+g(b-A\lambda)\right\},
\end{equation}
where $A\in \R^{m\times n}$, $b\in\R^m$ are given data. 
 We assume that 
$f$ and $g$ are lower semi-continuous and convex, $g$ is differentiable with gradient that is Lipschitz with modulus $L_g$. 
We can rewrite the above problem equivalently as
 $\min\left\{ f(\lambda)+g(\mu):\ A\lambda+\mu=b\right\}.$
The corresponding dual problem is 
\begin{equation*}
\label{dual1}
\min_\lambda\left\{ f^*(A^\top z)+g^*(z)-\< z,b\>\right\},
\end{equation*}
which is equivalent to 
\begin{equation}
\label{dual2}
\min\left\{ f^*(y)+g^*(z)-\< z,b\> \mid \ y-A^\top z=0\right\}.
\end{equation}

The following lemma shows the relation between (\ref{unconstraint}) and (\ref{dual2}).

\begin{lem}\label{pd}
If $(y^*,z^*,\lambda^*)$ is a $KKT$ solution to (\ref{dual2}),
then $\lambda^*$ is an optimal solution to (\ref{unconstraint}).
\end{lem}

\begin{proof}
From the optimality conditions of
\eqref{dual2}, we have $0\in \partial f^*(y^*)-\lambda^*$, 
$0\in \partial g^*(z^*)-b +A\lambda^*$ and $y^*-A^\top z^*=0$. Then $A^\top z^*\in A^\top \partial g(b-A \lambda^*)$, $y^*\in \partial f(\lambda^*)$, from which we deduce $0\in  \partial f(\lambda^*)-A^\top\partial g(b-A\lambda^*)$. Then $\lambda^*$ is an optimal solution to $\min\left\{ f(\lambda)+g(b-A\lambda)\right\},$ so $\lambda^*$ is an optimal solution to $(\ref{unconstraint})$. 
\end{proof}

By Proposition 12.60 of \cite{Via}, we know that $g^*$ is strongly convex with parameter $\sigma_g:=1/L_g.$ Thus, we can use Algorithm 1 to solve problem (\ref{dual2}). To avoid computing the conjugate function, we may use the following identity (see Theorem 14.3 in \cite{monotone})
due to Moreau to solve the subproblem.
\begin{prop}\label{prox}
Let $F$ be a convex lower semi-continuous function. For any $\alpha>0$, we have
$$ {\bf Prox}_{\alpha^{-1}F^*}(x)=x-\frac{1}{\alpha}{\bf Prox}_{\alpha F}(\alpha x).$$
\end{prop}

With Proposition~\ref{prox}, IADMM for solving (\ref{dual2}) is presented as follows.
\begin{algorithm}
	\renewcommand{\algorithmicrequire}{\textbf{Input:}}
	\renewcommand{\algorithmicensure}{\textbf{Output:}}
	\caption{IADMM for solving (\ref{unconstraint})}
	\label{alg:3}
	\begin{algorithmic}
		
		\STATE {\bf Initialization}: Given $ (x^1,\lambda^1)$ and constants $\gamma\in \big(1,\frac{1+\sqrt{5}}{2}\big)$, $\epsilon, \beta, \tau\in (0,1)$, choose the initial penalty parameter $\beta_1\in [\beta,\infty)$ and positive semidefinite matrix $Q$ such that 
		$AA^\top + Q \succ 0$. Set $P_k=0$, $Q_k=\beta_k Q$.
		\FOR{$k=1,2,\dots$}		
		\STATE 1. $y^{k+1}=A^\top z^k-\frac{\lambda^k}{\beta_k}-\frac{1}{\beta_k}\prox1_{\beta_k f}\left( \beta_k A^\top z^k-\lambda^k\right)$
		\STATE 2. $z^{k+1}= \arg\min_z\left\{ g^*(z)-\< z,b\>+\< \lambda^k,-A^\top z\>+\frac{\beta_k}{2}\ll y^{k+1}-A^\top z\ll^2+\frac{1}{2}\ll z-z^k\ll^2_{Q_k}\right\}$
		\STATE 3. $\lambda^{k+1}=\lambda^{k}+\gamma\beta_k({y^{k+1}-A^\top z^{k+1}})$
		\STATE 4. Choose $\beta_{k+1}\in \left[\max\{\beta,\tau\beta_k\},\ \sqrt{\beta_k^2
		+\frac{(1-\epsilon)\sigma_g\beta_k}{\lambda_{\max}(C^\top C+Q)}}\,\right]$
		\ENDFOR
	\end{algorithmic}  
\end{algorithm}

Note that to make subproblem in Step 2 of Algorithm \ref{alg:3} easier to solve, 
we may choose $Q:= \lambda_{\max}(AA^\top) I-AA^\top$. Then Step 2 becomes $$ z^{k+1}=a^k-\frac{1}{\beta_k\lambda_{\max}(AA^\top)}\prox1_{\beta_k\lambda_{\max}(AA^\top)g}\left(\beta_k\lambda_{\max}(AA^\top) a^k\right),$$ where $$ a^k=\frac{b+A\lambda^k+\beta_k\left( \lambda_{\max}(AA^\top) I-AA^\top\right) z^k+\beta_k A y^{k+1}}{\beta_{k} \lambda_{\max}(AA^\top)}.$$

\subsection{Strategies for updating $\beta_k$}
In this subsection, we consider how to update the penalty parameter $\beta_k$. Note that if we fix $\beta_k$ as a constant, then IADMM is equivalent to the traditional ADMM. Also, Xu's accelerated ADMM in \cite{XYY} is essentially IADMM with the following monotone updating strategy:
\begin{equation}\label{accstr}
\beta_1=\beta,\ \beta_{k+1}=\sqrt{\beta_k^2+\frac{(1-\epsilon)\sigma_g\beta_k}{\lambda_{\max}(C^\top C+Q)}},\forall k\in \mathbb{N}^+.
\end{equation}
It is called accelerated ADMM because the function value gap and primal
feasibility achieve the ergodic convergence rate of $O(1/n^2)$ as shown in Section~\ref{Sec-conv}. But as we will see later in the numerical experiments, a better strategy 
is to adaptively adjust the penalty parameter $\beta_k$ in IADMM based on 
 the ratio between the normalised primal feasibility $R_p^{k+1}$ and normalised dual feasibility $R_d^{k+1}$ for the computed iterate $(y^{k+1},z^{k+1},\lambda^{k+1})$, where 
\begin{eqnarray*}
&R_p^{k+1}:=\frac{\|By^{k+1}+Cz^{k+1}-b\|}{\max\left\{\|By^{k+1}\|,\|Cz^{k+1}
\|,\|b\|\right\}},\\[5pt]
&R_d^{k+1}:=\max\left\{ \frac{\| y^{k+1}-\prox1_f(y^{k+1}-B^\top \lambda^{k+1})\|}{\max\left\{ \|y^{k+1}\|,\|B^\top \lambda^{k+1}\|\right\}},\frac{\| z^{k+1}-\prox1_g(z^{k+1}-C^\top \lambda^{k+1})\|}{\max\left\{\|z^{k+1}\|,\|C^\top\lambda^{k+1}\|\right\}} \right\}.
\end{eqnarray*}
The KKT residue of the computed $(y^{k+1},z^{k+1},\lambda^{k+1})$ is defined to be $\max\{R_p^{k+1},R_d^{k+1}\}$. Our adaptive strategy is that after iteration $k$, we check the ratio between $R_p^{k+1}$ and $R_d^{k+1}$, and  update $\beta_k$ as follows:
\begin{equation}\label{adapstr}
\beta_{k+1}:= \begin{cases} \sqrt{\beta_k^2+\frac{(1-\epsilon)\sigma_g\beta_k}{\lambda_{\max}(C^\top C+Q)}} & R_p^{k+1}>R_d^{k+1},\\[5pt]
\max\left\{ \beta, \beta_k/1.5\right\} & R_p^{k+1}<R_d^{k+1}/10,\\
\beta_k & R_d^{k+1}/10 \leq R_p^{k+1}\leq R_d^{k+1}.
\end{cases}
\end{equation}

It is easy to see that this strategy belongs to the framework of Algorithm 1. Note that updating the penalty according to the ratio between primal and dual feasibility is popular in the literature (see \cite{He-adapt, xu2017adaptive2}). However, our adaptive strategy has guaranteed
convergence. The adaptive strategy for IADMM is not necessarily restricted to the one presented in (\ref{adapstr}). One can choose any other adaptive strategy, and as long as the new penalty parameter lies in the interval in Step 4 of Algorithm 1, the convergence is guaranteed. In practice, we choose  $\gamma=1.618$ and $\epsilon=10^{-4},$ $\beta=10^{-6}$. 

\subsection{Strategies for partial proximal point method}\label{Subsec-PPPM}
In this section, we consider the implementation of Algorithm~\ref{alg:PPPM}. Because the convergence analysis in Theorem~\ref{conv2} requires $\sigma_k$ to be monotonically decreasing, we will choose $\sigma_k=\max\{1/2^k,10^{-6}\}.$ A geometrically decreasing proximal parameter usually reduces the number of outer iterations in the partial proximal point method. Let $\widehat{R}_p^k,$ $\widehat{R}_d^k$ be the primal and dual KKT residue of the $k$th subproblem (\ref{prox-prob}). Let $R_p^k,$ $R_d^k$ be the KKT residue of the original problem (\ref{eq-prob}). Because problems (\ref{eq-prob}) and (\ref{prox-prob}) have the same constraints, we get $\widehat{R}_p^k=R_p^k.$ Also, when the subproblem (\ref{prox-prob}) is solved exactly, we will have that $R_p^k=\widehat{R}_p^k=\widehat{R}_d^k=0$ but $R_d^k$ may not equal to zero. This is because problem (\ref{eq-prob}) is generally not equivalent to subproblem (\ref{prox-prob}). 
We stop the IADMM for solving the subproblem when the following conditions
are satisfied:
\begin{equation}\label{st-cri}
R_p^k<R_d^k/10,\ \max\left\{\widehat{R}_p^k,\widehat{R}_d^k\right\}<1/(10k^3).
\end{equation}
The first condition in (\ref{st-cri}) uses the relation between the primal and dual KKT residues. It has been used in augment Lagrangian method \cite{eckstein2013practical}. The second condition in (\ref{st-cri}) ensures that the KKT residues of the subproblems tend to zero rapidly with the rate $O(1/k^3).$ It corresponds to the condition of the error term in Theorem~\ref{conv2}. 

We should add that the initial penalty parameter $\beta_1$ used by the IADMM 
to solve the $k$-th PPPM subproblem in Algorithm 2 is adaptively adjusted as follows.
In detail, if the penalty parameter of the previous IADMM inner loop keeps increasing in the last few iterations, it is likely to be too small and will continue to increase in the new IADMM inner loop. In this case, we set the initial penalty parameter to be $\eta$ times the last penalty parameter of the previous IADMM inner loop for some constant $\eta>1.$ Otherwise, we simply choose the initial penalty parameter to be the same as the last penalty parameter of the previous IADMM inner loop. By doing so, Algorithm~\ref{alg:PPPM} can increase the 
penalty parameter drastically between different outer iterations and refine it by the scheme in (\ref{adapstr}) at each inner loop. From our numerical experiment, Algorithm~\ref{alg:PPPM} can always identify a good penalty parameter after only a few outer iterations and becomes stable after that. 

\section{Numerical experiments}\label{Sec-numer}

In this section, we apply our algorithms to different problems to demonstrate the robustness, convergence results and efficiency. Although in our theoretical analysis, the proximal term of $y$ can be nonzero, we simply choose $P=0$ in the following numerical experiments. This is because the ADMM subproblem for $y$ already has a closed form solution and we don't have to add
a proximal term to simplify it. All the experiments are run using {\sc Matlab} R2021b on a Workstation with a Intel(R) Xeon(R) CPU E5-2680 v3 @ 2.50GHz Processor and 128GB RAM. 

\subsection{Testing the robustness of IADMM}

In this section, we test several regression problems to verify the robustness of adaptive IADMM.

\begin{exmp}
{\bf Total variation regularized least squares problem.}
\end{exmp}
We consider following problem, which is considered in example 6.3 of \cite{Kimacc}:
\begin{equation}\label{TV}
\min\left\{ \|Hx-b\|^2/2+\gamma \sum_{i=1}^{n-1}|x_{i+1}-x_i|:\ x\in \R^n\right\},
\end{equation}
where $H\in \R^{r\times n}$, $b\in \R^r$. To solve (\ref{TV}), we apply IADMM to its dual problem as in Subsection~\ref{Subsec-LASSO}. In this case, $f(x)=\gamma \sum_{i=1}^{n-1}|x_{i+1}-x_i|$ and $g(x) = \|x\|^2/2.$ To compute the proximal mapping of $f$, we use Condat's direct algorithm in \cite{condat}. We use synthetic data similar to example 6.3 in \cite{Kimacc}. We randomly generate the matrix $H$ as \texttt{H = randn(r,n)/sqrt(n)} and $b$ as \texttt{b = H*x\_true+nf*randn(r,1)}, where \texttt{x\_true} is a randomly generated vector such that $\left(  x_2-x_1,x_3-x_2, \ldots ,x_n-x_{n-1}\right)$ only has a few non-zero elements. We choose the noise  factor \texttt{nf = 1/norm(H,'fro')}.
\begin{exmp}
{\bf Elastic net regularized support vector machine}
\end{exmp}
We consider the following problem which is considered in section 3.3 of \cite{XYY}:
\begin{equation}\label{SVM}
\min\left\{\frac{1}{m}e^\top [y]_++\mu_1 \|x\|_1+\frac{\mu_2}{2}\|x\|^2:\ Bx+y=e,\ x\in \R^p,\ y\in \R^m\right\},
\end{equation}
where $B\in \R^{m\times p}$ and $e\in\R^m$ is the vector of all ones. We choose $\mu_1=\mu_2=0.01$ and generate synthetic data matrix $B$ in the same way as section 3.3 of \cite{XYY} (also see section 3.1.1 of \cite{XYYela}).

\begin{exmp}
{\bf Elastic net regularized problem with square-root loss}
\end{exmp}
We consider the following problem which is considered in section 5.2. of \cite{QTD2}:
\begin{equation}\label{L2}
\min\left\{ \|x\|_2+\mu_1\|y\|_1+\frac{\mu_2}{2}\|y\|^2:\ -x+By=c,\ x\in \R^n,\ y\in \R^p\right\},
\end{equation}
where $B\in \R^{n\times p}$ and $c\in \R^n.$ We choose $\mu_1=0.01,$ $\mu_2=0.1$ and generate synthetic data matrix $B,$ $c$ in the same way as section 5.2 of \cite{QTD2}.

For all the above problems, we test three algorithms: traditional ADMM with fixed-penalty parameter, IADMM with adaptive strategy (\ref{adapstr}) and ADMM with a heuristic adaptive strategy (which we denote as Heu-ADMM) as follows:
\begin{equation}\label{adapheu}
\beta_{k+1}:= \begin{cases} 1.5\beta_k & R_p^{k+1}>10R_d^{k+1},
\\[5pt]
\beta_k/1.5 & R_p^{k+1}<R_d^{k+1}/10, \\
\beta_k & R_d^{k+1}/10\leq R_p^{k+1}\leq 10R_d^{k+1}.
\end{cases}
\end{equation}

We choose the initial penalty parameter to be from $10^{-5}$ to $10^{5}$. We stop the algorithms when $\max\{R_p^{k+1},R_d^{k+1}\}<10^{-5}$ or when the maximum iteration number is reached.

\begin{figure}
		\begin{minipage}{1\linewidth}
		\vspace{3pt}
		\centerline{\includegraphics[height=15cm,width=18cm]{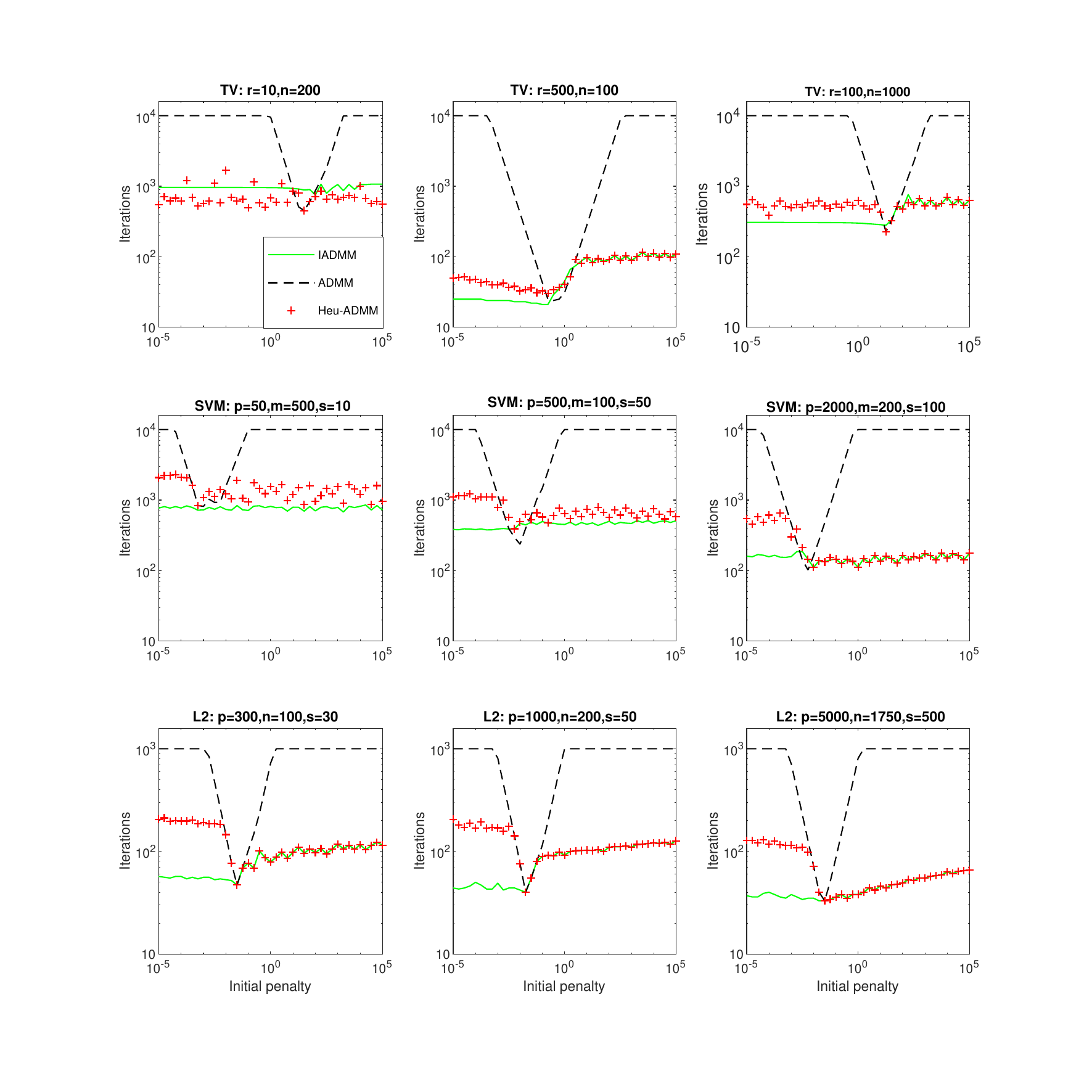}}
		\vspace{-10mm}
	\end{minipage}
	\caption{Results for Example 5.1, 5.2 and 5.3. ``IADMM" is adaptive ADMM with strategy (\ref{adapstr}), ``ADMM" is ADMM with a fixed-penalty parameter, ``Heu-ADMM" is adaptive ADMM with strategy (\ref{adapheu}). $``{\bf s}"$ is a parameter for the generating data, readers may refer to \cite{QTD2,XYY}.
}
	\label{fig1}
\end{figure}

The results of this experiment are shown in Figure~\ref{fig1}. The $x-$axis is the initial penalty parameter and the $y-$axis is the number of iterations that the algorithm needs to achieve the accuracy of $10^{-5}.$  From the plots in Figure \ref{fig1}, we can see that the fixed parameter ADMM is very fast once the penalty parameter is optimally tuned. However, its speed is very sensitive to the initial penalty parameter. However, IADMM and Heu-ADMM can solve all the problems efficiently, without being affected seriously by the initial parameter, because they can adaptively tune their penalty parameters. Apart from the convergence guarantee, our adaptive strategy (\ref{adapstr}) is also more stable than the heuristic strategy (\ref{adapheu}). Moreover, when the initial penalty parameter is too small, IADMM is usually faster than Heu-ADMM. One possible reason is that in the strategy (\ref{adapstr}), $\beta_k$ has a super-exponential growth when it is too small. 

\subsection{Verifying the convergence rate of IADMM}

In this section, we consider the following dense convex quadratic program, which is considered in section 5.1 of \cite{QTD2}.

\begin{exmp}
{\bf Dense convex quadratic programs}
\end{exmp}
\begin{equation}\label{QP}
\min\left\{ \frac{1}{2}y^\top H y+h^\top y:\ y\in \R^m,\ a\leq By\leq b\right\},
\end{equation}
where $a,b\in \R^n,$ $h\in \R^m,$ $B\in \R^{n\times m}$ and $H\in \mathbb{S}^m$ is a positive definite matrix. After introducing another variable $z\in \R^n$, (\ref{QP}) becomes:
\begin{equation}\label{QP1}
\min\left\{ \frac{1}{2}y^\top Hy+h^\top y+\delta_{[a,b]}\left(z\right):\  By-z=0,\; y\in \R^m\right\},
\end{equation}
where $\delta_{[a,b]}(.)$ is the indicator function of the 
set $\left\{ x\in \R^n:\ a\leq x\leq b\right\}.$
Since $\frac{1}{2}y^\top H y + h^\top y$ is strongly convex, we can use IADMM to solve (\ref{QP1}). We generate data in the same way as section 5.1 of \cite{QTD2}. We choose $n=m=2000$ and the smallest eigenvalue of $H$ is 1, $0.5$, $0.1$ for different choices of $H$. Moreover, we choose $B$ as a matrix with full row rank. Therefore, from Theorem~\ref{linearc}, IADMM has linear convergence as long as the penalty parameters are bounded.

We will compare IADMM with Tran-Dinh and Zhu's accelerated ADMM in section 4.4 of \cite{QTD}, Kim's accelerated ADMM in section 6.4 of \cite{Kimacc} and Xu's accelerated ADMM in \cite{XYY}. We call them ``acc1-ADMM", ``acc2-ADMM", ``acc3-ADMM", respectively. For acc1-ADMM, we update the
parameters as suggested in \cite{QTD}\footnote{The parameters used in acc1-ADMM are quite different from the traditional ADMM, so we omit the details here. Readers may refer to \cite{QTD} (31) case 2 and section 5.1 case 2  for details.}. For acc2-ADMM, we choose the 
initial penalty parameter to be ${25}/({2\|B\|_2^2})$, which is nearly optimally-tuned. For acc3-ADMM, we choose $\beta_1=\beta={1}/({2\|B\|_2^2})$ and use the strategy 
presented in (\ref{accstr}) to update $\beta_k.$ For IADMM, we choose $\beta_1=\beta={1}/({2\|B\|_2^2})$ and use the strategy in (\ref{adapstr}) to update $\beta_k.$

\begin{figure}
		\begin{minipage}{1\linewidth}
		\vspace{3pt}
		\centerline{\includegraphics[height=15cm,width=20cm]{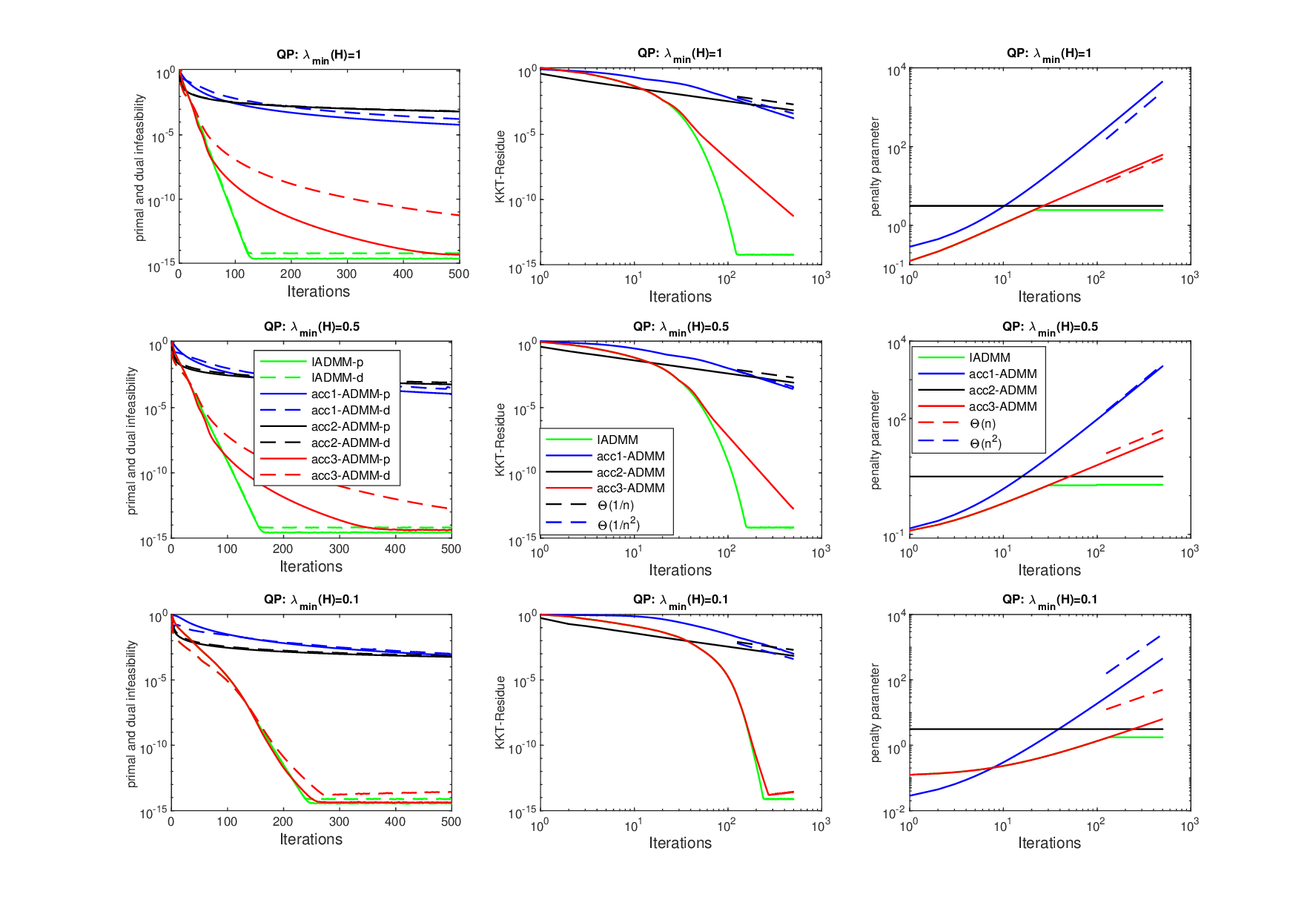}}
          \vspace{-10mm}
	\end{minipage}
	
	\caption{Results for the quadratic program \eqref{QP1} with different $H$. In the left panel, the curves 
      {\tt ***}-p and {\tt ***}-d correspond to the primal feasibility and dual feasibility, respectively.
}
	\label{fig2}
\end{figure}

The results of this experiment are shown in Figure~\ref{fig2}. The first column of plots is on the primal and dual infeasibility. We use the matlab function \texttt{semilogy} to visualize the linear convergence of IADMM. The second column of plots is on the KKT-residue. We use the matlab function \texttt{loglog} to show the sub-linear convergence rate of different types of ADMM algorithms. The third column of plots is on the penalty parameters. We use \texttt{loglog} to show the growth rate of the penalty parameters. From the plots of Figure~\ref{fig2}, we can see that IADMM is the most efficient
algorithm. The convergence rate of acc2-ADMM is close to $O(1/n)$, which is consistent with the theory in \cite{Kimacc}. This example also implies that the convergence rate of $O(1/n)$ for acc2-ADMM cannot be improved even if one part of the objective function is strongly convex. We also see that the convergence rate of acc1-ADMM is close to $O(1/n^2),$ which verifies the theory in \cite{QTD}. For acc3-ADMM, its convergence rate is 
surprisingly much faster
than the ergodic convergence rate of $O(1/n^2)$ proved in Section~\ref{Sec-conv}. 
 IADMM is the only algorithm which has linear convergence. This is consistent with Theorem~\ref{linearc} because from the third plot, the penalty parameter for IADMM is bounded. However, since the penalty parameter for acc3-ADMM tends to infinity, the condition of Theorem~\ref{linearc} is not satisfied and acc3-ADMM does not exhibit linear convergence. Moreover, it is easy to see that when the penalty parameter is too large, the primal feasibility of acc3-ADMM is much smaller than the dual feasibility of acc3-ADMM\footnote{Note that the acc1-ADMM is quite different from the traditional ADMM. Its primal and dual feasibility is close to each other even if its penalty parameter increases rapidly.}. This is because the penalty on the dual feasibility becomes weaker with 
a larger $\beta_k$. That's why we had better choose the penalty parameter adaptively.

We have also tested \eqref{QP1} on other random instances, but the 
behaviours of acc1-ADMM, acc2-ADMM, acc3-ADMM and IADMM are similar
to the ones presented in Figure \ref{fig2}.

\subsection{Testing the efficiency of partial proximal point method for solving non-semi-strongly
convex problems}
In this section, we verify the efficiency of Algorithm~\ref{alg:PPPM}. We consider the following two examples.

\begin{exmp}
{\bf Standard convex quadratic programming}
\begin{equation}\label{STCVXQP}
\min\left\{ \frac{1}{2}x^\top Qx+c^\top x:\ Ax=b,\ l\leq x\leq u,\ x\in \R^n\right\},
\end{equation}
where $Q\in \mathbb{S}^n_+,$ $c\in \R^n,$ $A\in \R^{m\times n},$ $b\in \R^m,$ $l,u\in \(\R\cup\{\pm \infty\}\)^n.$ 
\end{exmp}
Problem (\ref{STCVXQP}) is the standard form of a convex quadratic programming problem. We 
reformulate problem (\ref{STCVXQP}) into the following problem:
\begin{equation}\label{STCVXQP}
\min\left\{ \frac{1}{2}z^\top Qz+c^\top z+\delta_{[l,u]}(z)+\delta_\Omega(y):\ y=z,\  y,z\in \R^n\right\},
\end{equation}
where $\Omega = \{ y\in \R^n\mid Ay=b\}.$
The above problem can be solved by ADMM-type algorithms. Because the matrix $Q$ may not be positive definite, we use the partial proximal point method to solve it. For the IADMM subproblem with respect to $z,$ we use the semi-smooth Newton method to solve its dual problem. In detail, after some simplification, we subproblem with respect to $z$ can be written as 

\begin{equation}\label{SSNQP}
\min\left\{ z^\top Q z/2+\beta \|z-\hat{z}\|^2/2+\delta_{[l,u]}(z):\ z\in \R^n \right\},
\end{equation}
for some $\beta>0$ and $\hat{z}\in \R^n.$ Without loss of generality, we assume that $Q$ is positive definite. This is because otherwise we can reformulate the above problem as
\begin{equation}\label{SSNQP1}
\min\left\{ z^\top \(Q+\epsilon I\) z/2+\(\beta-\epsilon\) \|z-\beta\hat{z}/(\beta-\epsilon)\|^2/2+\delta_{[l,u]}(z):\ z\in \R^n \right\}
\end{equation}
for some small constant $\epsilon>0.$ Problem (\ref{SSNQP}) and (\ref{SSNQP1}) only differ by a constant. Since $Q$ is positive definite, we can apply Cholesky factorization to decompose it as $Q:=LL^\top$ for some lower triangular matrix $L$. Note that we only have to do this once at the beginning because $Q$ is a constant matrix. Thus, problem (\ref{SSNQP}) can be further written as
\begin{equation}\label{SSNQP2}
\min\left\{ \|L^\top z\|^2/2+\beta \|z-\hat{z}\|^2/2+\delta_{[l,u]}(z):\ z\in \R^n \right\}.
\end{equation}
After introducing another variable $y=L^\top z,$ problem (\ref{SSNQP2}) becomes
\begin{equation}\label{SSNQP3}
\min\left\{ \|y\|^2/2+\beta \|z-\hat{z}\|^2/2+\delta_{[l,u]}(z):\ y=L^\top z,\ y\in \R^n,\ z\in \R^n \right\},
\end{equation}
whose dual problem is
\begin{equation}\label{SSNQPD}
\min\left\{ \beta\( \|\hat{z}-L\lambda/\beta\|^2-\dist\([l,u],\hat{z}-L\lambda/\beta\)^2 \)/2+\|\lambda\|^2/2:\ \lambda\in \R^n \right\},
\end{equation}
which is an unconstraint optimization problem with the objective function being strongly convex and Lipschitz continuous differentiable. One can use semi-smooth newton method to solve (\ref{SSNQPD}) efficiently. In our experiment, we set the tolerance of the gradient norm to be less than $10^{-10}$ to terminate the semi-smooth Newton method.

We compare Algorithm~\ref{alg:PPPM} with the barrier method of  Gurobi \cite{gurobi} 9.5.2  with 2 threads on the modified Maros-M\'esz\'aros benchmark dataset \cite{maros1999repository}. Because the matrix $Q$ of some of the instances in \cite{maros1999repository} is diagonal, 
Gurobi can solve such problems very efficiently by making use of this highly special structure. In order to make the problems more challenging, we add a small perturbation to $Q$ whenever it is diagonal as follows:
\begin{center}
\texttt{B=randn(n,n); BB = B*B'; Q = Q+1e-6*norm(Q,'fro')*BB/norm(BB,'fro').}
\end{center}
We use ``*'' to indicate the perturbed instances in the following table. After the random perturbation, $Q$ becomes a dense matrix. We only test the problems in \cite{maros1999repository} with standard form and $1000\leq n\leq 10000.$ When $n$ is greater than $10000,$ the dense Cholesky factorization is too expensive for both Algorithm~\ref{alg:PPPM} and Gurobi. We use the following standard KKT residue for convex QP problem to measure the accuracy:
\begin{equation}\label{QPKKT}
{\rm Resp}:=\frac{\| Ax-b\|}{1+\|b\|},\ {\rm Resd}:=\frac{\|{\rm Proj}_{[l,u]}\(x-\(Qx+c-A^\top \lambda\)\)-x\|}{1+\|Qx+c\|}
\end{equation}

 Since Gurobi use a different stopping criterion instead of KKT residue, we first use Gurobi to solve one problem with tolerance $10^{-5}$ to get the outcome. Then we compute the KKT residue i.e., $\max\{{\rm Resp},{\rm Resd}\}$ of the outcome, say $\alpha$, and set the tolerance of PPPM to be $\min\{10^{-5},\alpha\}.$ By doing so, the outcome of our algorithm will be more accurate than Gurobi and the comparison of running time will be fair.

\begin{center}
\begin{footnotesize}
\begin{longtable}{|c|c|cccc|l|}
\caption{Comparison of PPPM(IADMM), and Gurobi for modified 
Maros-M\'esz\'aros convex QP problem; $(n,m)$ denotes the number of variables and affine constraints.}
\label{MMQP}
\\
\hline
Problem & Algorithm & Resp & Resd & Fval & Time [s] \\ \hline
\endhead
 AUG3DCQP* & PPPM & 9.61e-06 & 1.81e-06 & -9.4314271e+02& 4.06e+00 \\
(3873,1000) & Gurobi & 4.64e-11 & 4.64e-04 & -9.4313492e+02& 9.93e+01 \\
 \hline 
 AUG3DQP* & PPPM & 9.38e-06 & 4.06e-06 & -6.6126864e+02& 1.44e+01 \\
 (3873,1000) & Gurobi & 1.82e-14 & 1.31e-03 & -6.6125924e+02& 9.95e+01 \\
 \hline 
  CONT-050* & PPPM & 7.79e-07 & 8.39e-07 & -4.5638482e+00& 2.77e+02 \\
 (2597,2401) & Gurobi & 2.45e-12 & 8.51e-07 & -4.5638481e+00& 2.41e+01 \\
 \hline 
 CVXQP1\_L & PPPM & 4.58e-09 & 6.78e-09 & 1.0870480e+08& 4.71e+00 \\
 (10000,5000) & Gurobi & 7.37e-09 & 4.40e-09 & 1.0870481e+08& 1.56e+01 \\
 \hline 
  CVXQP1\_M & PPPM & 1.95e-08 & 6.07e-08 & 1.0875116e+06& 1.51e-01 \\
 (1000,500) & Gurobi & 6.90e-08 & 2.90e-08 & 1.0875127e+06& 1.88e-01 \\
 \hline 
 CVXQP2\_L & PPPM & 1.53e-09 & 8.70e-09 & 8.1842458e+07& 5.50e+00 \\
 (10000,2500) & Gurobi & 2.85e-12 & 8.90e-09 & 8.1842458e+07& 7.12e+00 \\
 \hline 
  CVXQP2\_M & PPPM & 2.42e-08 & 1.22e-07 & 8.2015543e+05& 2.16e-01 \\
 (1000,250) & Gurobi & 1.85e-10 & 1.24e-07 & 8.2015543e+05& 1.01e-01 \\
 \hline 
  CVXQP3\_L & PPPM & 9.90e-10 & 8.38e-09 & 1.1571110e+08& 6.25e+00 \\
 (10000,7500) & Gurobi & 8.45e-09 & 5.54e-09 & 1.1571112e+08& 2.03e+01 \\
 \hline 
  CVXQP3\_M & PPPM & 7.54e-08 & 2.04e-07 & 1.3628288e+06& 2.19e+00 \\
 (1000,750) & Gurobi & 2.18e-07 & 4.41e-08 & 1.3628301e+06& 2.60e-01 \\
 \hline 
  HUES-MOD* & PPPM & 8.98e-06 & 2.81e-06 & 3.4824562e+07& 3.01e+01 \\
 (10000,2) & Gurobi & 4.00e-15 & 2.59e-05 & 3.4824489e+07& 1.27e+03 \\
 \hline 
  HUESTIS* & PPPM & 4.38e-08 & 1.38e-08 & 3.4824489e+11& 3.34e+01 \\
 (10000,2) & Gurobi & 3.66e-15 & 5.53e-08 & 3.4824489e+11& 1.74e+03 \\
 \hline 
  QSCSD6 & PPPM & 9.95e-06 & 8.17e-06 & 5.0808222e+01& 1.00e+00 \\
 (1350,147) & Gurobi & 2.99e-13 & 1.08e-05 & 5.0808253e+01& 1.48e-02 \\
 \hline 
  QSCSD8 & PPPM & 9.38e-06 & 6.29e-06 & 9.4076384e+02& 1.11e+00 \\
 (2750,397) & Gurobi & 2.68e-14 & 3.30e-05 & 9.4076594e+02& 2.40e-02 \\
 \hline 
  STCQP1 & PPPM & 7.14e-07 & 1.59e-06 & 1.5514388e+05& 1.76e+00 \\
 (4097,939) & Gurobi & 3.70e-11 & 1.60e-06 & 1.5514356e+05& 1.73e-01 \\
 \hline 
  STCQP2 & PPPM & 1.11e-06 & 9.89e-06 & 2.2327448e+04& 6.63e-01 \\
 (4097,2052) & Gurobi & 3.57e-13 & 1.69e-05 & 2.2327319e+04& 3.27e-01 \\
 \hline 
 \end{longtable}
\end{footnotesize}
\end{center}

From Table~\ref{MMQP}, we can see that Algorithm~\ref{alg:PPPM} can solve all the instances to the required accuracy. Gurobi does not reach the tolerance for some instances because it uses a different stopping criterion. Algorithm~\ref{alg:PPPM} is faster than Gurobi for more than half of the instances. 
This shows that when $Q$ is not well structured, our algorithm (PPPM with IADMM as subproblem solver) is efficient
enough to compete favourably with Gurobi,  which is a well-developed solver for convex QP problems.


\medskip
Next we consider another class of non-semi-strongly convex problems, rank lasso problems,
 to evaluate
the efficiency of IADMM within the PPPM framework.

\begin{exmp}
{\bf Rank LASSO problem}
\begin{equation}\label{RLASSO}
\min\left\{ \frac{2}{n(n-1)}\sum_{1\leq i<j\leq n}\left| (b_i-a_i^\top x)-(b_j-a_j^\top x) \right|+\lambda \|x\|_1:\ x\in \R^m\right\}.
\end{equation}
\end{exmp}

Problem (\ref{RLASSO}) was proposed by Wang et al. in \cite{wang2020tuning}. Its computational aspect is recently studied in \cite{tang2021proximal,bai2022highly}. Let $A:=\(a_1,a_2,\ldots,a_n\)^\top$ and $b=\(b_1,b_2,\ldots,b_m\)^\top.$ Problem (\ref{RLASSO}) is also called tuning-free robust regression because its regularization parameter has the following formula (see (7) of \cite{wang2020tuning}).

\begin{equation}\label{eq_lam}
\lambda=cG^{-1}_{\|S_n\|_\infty}(1-\alpha_0),
\end{equation}
where $\alpha_0=0.1,$ $c=1.1$ and $G^{-1}_{\|S_n\|_\infty}(1-\alpha_0)$ denotes the $(1-\alpha_0)-$quantile of the distribution of $\|S_n\|_\infty.$ Here $S_n=-2A^\top \xi/n(n-1)$ and $\xi=2r-(n+1)$ with $r$ following the uniform distribution on the permutations of the integers $\{1,2,\ldots,n\}.$ Problem (\ref{RLASSO}) can be written as the following problem:
\begin{equation}\label{RLASSO1}
\min\left\{ h(y)+\lambda \|z\|_1:\ Az-b=y,\ z\in \R^m,\ y\in \R^n\right\},
\end{equation}
where $h(y):=\frac{2}{n(n-1)}\sum_{1\leq i<j\leq n}|y_i-y_j|.$ Because problem (\ref{RLASSO1}) is not semi-strongly convex, we use the partial proximal point method to solve it. When we use IADMM to solve the proximal subproblem of (\ref{RLASSO1}), updating $z$ corresponds to solving the following problem:
\begin{equation}\label{Update_z}
\min\left\{ \lambda\|z\|_1+\beta\| Az-b-y-\lambda/
\beta \|^2/2+\sigma \|z-z^k\|^2/2:\ z\in \R^m \right\},
\end{equation}
which is similar to the proximal point subproblem of the square LASSO problem. We can
use the semi-smooth Newton method \cite{li2018highly} to solve its dual problem efficiently. Updating $y$ can be done easily by computing the proximal mapping of $h(y).$ We can use the direct solver developed in \cite{lin2019efficient} to solve it in nearly linear time. Because updating the variable $z$ doesn't have a closed form solution for a general matrix $A,$ we compare Algorithm~\ref{alg:PPPM} two variants of ADMM's. The first one is linearized ADMM,
which is presented in Algorithm \ref{alg:LADMM}.

\begin{algorithm}[h]
	\renewcommand{\algorithmicrequire}{\textbf{Input:}}
	\renewcommand{\algorithmicensure}{\textbf{Output:}}
	\caption{LADMM}
	\label{alg:LADMM}
	\begin{algorithmic}
		\STATE {\bf Initialization}: Given $ (x^1,\lambda^1)$ and constants $\gamma\in \big(1,\frac{1+\sqrt{5}}{2}\big).$		\FOR{$k=1,2,\dots$}		
		\STATE 1. $y^{k+1}=\arg\min_y\left\{ h(y)+\frac{\beta_k}{2}\| Az^k-y-b-\lambda^k/\beta_k \|^2 \right\}$
		\STATE 2. $z^{k+1}= \arg\min_z\left\{ \lambda\|z\|_1+\frac{\beta_k}{2}\| Az-y^{k+1}-b+\lambda^k\beta_k \|^2+\frac{\beta_k}{2}\|z-z^k\|^2_{\lambda_{\max}(A^\top A)I-A^\top A} \right\}$
		\STATE 3. $\lambda^{k+1}=\lambda^{k}-\gamma\beta_k(A^\top z^{k+1}-y^{k+1}-b)$
		\STATE 4. Update $\beta_{k+1}$ from (\ref{adapheu})
		\ENDFOR
	\end{algorithmic}  
\end{algorithm}

In Step 2 of Algorithm~\ref{alg:LADMM}, the proximal term with the weighted matrix $\beta_k (\lambda_{\max}(A^\top A)I-A^\top A)$ 
can simplify the subproblem into computing the proximal mapping of $\lambda\|\cdot\|,$ which has a closed form solution. The second variant is that we introduce another variable $u\in \R^m$ and write problem (\ref{RLASSO1}) equivalently as follows:
\begin{equation}\label{RLASSO2}
\min\left\{ h(y)+\lambda \|z\|_1:\ Au-b=y,\ u=z,\ z\in \R^m,\ y\in \R^n,\ u\in \R^m\right\}.
\end{equation}
We can use the  ADMM presented in Algorithm \ref{alg:ADMM} to solve (\ref{RLASSO2}).

\begin{algorithm}[h]
	\renewcommand{\algorithmicrequire}{\textbf{Input:}}
	\renewcommand{\algorithmicensure}{\textbf{Output:}}
	\caption{ADMM}
	\label{alg:ADMM}
	\begin{algorithmic}
		\STATE {\bf Initialization}: Given $ (y^1,z^1,u^1,\lambda^1,\mu^1)$ and constants $\gamma\in \big(1,\frac{1+\sqrt{5}}{2}\big).$		\FOR{$k=1,2,\dots$}		
		\STATE 1. $u^{k+1}=\arg\min_y\left\{ \frac{\beta_k}{2}\| Au-b-y^k-\lambda^k/\beta_k \|^2+\frac{\beta_k}{2}\| u-z^k-\mu^k/\beta_k \|^2 \right\}$
		\STATE 2. 
		\begin{eqnarray*}
		(y^{k+1},z^{k+1})= \arg\min_{(y,z)}
		\left\{ 
		\begin{array}{l} h(y)+\frac{\beta_k}{2}\| Au^{k+1}-b-y-\lambda^k/\beta_k \|^2
		\\
		+\lambda \|z\|_1+\frac{\beta_k}{2}\| u^{k+1}-z-\mu^k/\beta_k \|^2 
		\end{array} \right\}
		\end{eqnarray*}
		\STATE 3. $\lambda^{k+1}=\lambda^{k}-\gamma\beta_k(A^\top u^{k+1}-y^{k+1}-b),$ $\mu^{k+1}=\mu^k-\gamma\beta_k(u^{k+1}-z^{k+1})$
		\STATE 4. Update $\beta_{k+1}$ from (\ref{adapheu})
		\ENDFOR
	\end{algorithmic}  
\end{algorithm}

In Algorithm~\ref{alg:ADMM}, Step 1 is essentially solving a positive definite linear system with the coefficient matrix $I_m+A^\top A.$ Note that we can apply the Sherman-Morrison-Woodbury formula \cite{golub1996cf} when $n$ is smaller than $m.$ We also store the Cholesky decomposition of $I+AA^\top$ to use it to solve the linear system in Step 1. Step 2 is 
equivalent to computing the proximal mappings of $h(y)$ and $\lambda \|z\|_1$ independently. Apart from ADMM, we can also formulate problem (\ref{RLASSO}) as a linear programming problem as mentioned in section 4.2 of \cite{wang2020tuning}. However, the linear programming formulation has huge number of variables and constraints. Thus we only test the linear programming model for small problems. 

For ADMM solvers, we terminate the algorithms when the KKT residue is smaller than $10^{-5}.$ We choose the initial penalty parameter to be $\max\{100/\lambda_{\max}(B)^2,10^{-6}\}.$ We set the maximum running time to be 3600s. For the linear programming model, we use
the barrier method in Gurobi 9.5.2 with 2 threads. We also set the tolerance to be $10^{-5}$ and the maximum running time to be 3600s;

We first consider synthetic dataset. We generated the data $A$ and $b$ in the same way as mentioned in example 1 of \cite{wang2020tuning}. In detail, the rows of $A$ are generated from a $p-$dimensional multivariable normal distribution $N_p(0,\Sigma),$ where the covariance matrix satisfies $\Sigma_{i,j}=0.5$ for $i\neq j$ and $\Sigma_{i,j}=1$ for $i=j.$ The ground truth $\hat{x}:=\(\sqrt{3},\sqrt{3},\sqrt{3},0,\ldots,0\)^\top.$ We generate $b$ as $b=A\hat{x}+\epsilon,$ where $\epsilon_i$ satisfies six different distributions: (1) $N(0,0.25)$; (2) $N(0,1)$; (3) $N(0,2)$; (4) $0.95N(0,1)+0.05N(0,100)$ (denoted by MN); (5) $\sqrt{2}t(4),$ where $t(4)$ denotes the $t$ distribution with 4 degree of freedom; (6) Cauchy(0,1). We choose $\lambda$ from the formula (\ref{eq_lam}) by randomly generating 1000 permutations of the integers $\{1,2,\ldots,n\}$ and computing the approximate $(1-\alpha_0)-$quantile. 

We first test small problems where $n=200$ and $m=1000.$ In this case, we may compared the above mentioned four algorithms. 

\begin{center}
\begin{footnotesize}
\begin{longtable}{|c|c|ccccc|l|}
\caption{Comparison of PPPM(IADMM), LADMM, ADMM and Gurobi for randomly generated rank LASSO problem. $n=200,$ $m=1000.$ For Algorithm~\ref{alg:PPPM}, ``Iter" means number of ADMM iterations.}
\label{T_small}
\\
\hline
Problem & Algorithm & Resp & Resd & Iter& Fval & Time [s] \\ \hline
\endhead
n = 200 & PPPM & 4.18e-06 & 7.67e-06 & 80 & 1.9747008e+00& 1.60e-01 \\
m = 1000 & LADMM & 9.92e-06 & 9.79e-06 & 3357 & 1.9747001e+00& 2.07e+00 \\
$\kappa$ = 0.32 & ADMM & 8.35e-06 & 1.00e-05 & 4206 & 1.9747004e+00& 3.13e+00 \\
$N(0,0.25)$ & Gurobi & 9.99e-19 & 4.74e-06 & 28 & 1.9747009e+00& 6.13e+01 \\
\hline
n = 200 & PPPM & 5.84e-06 & 9.98e-06 & 138 & 2.8465270e+00& 1.69e-01 \\
m = 1000 & LADMM & 9.17e-06 & 9.99e-06 & 3560 & 2.8465223e+00& 1.96e+00 \\
$\kappa$ = 0.32 & ADMM & 9.95e-06 & 9.27e-06 & 178321 & 2.8465172e+00& 1.31e+02 \\
$N(0,1)$ & Gurobi & 6.30e-18 & 1.39e-06 & 22 & 2.8465242e+00& 4.99e+01 \\
\hline
n = 200 & PPPM & 7.45e-06 & 9.46e-06 & 154 & 4.0089871e+00& 1.87e-01 \\
m = 1000 & LADMM & 9.98e-06 & 1.00e-05 & 6154 & 4.0089518e+00& 3.60e+00 \\
$\kappa$ = 0.32 & ADMM & 9.72e-06 & 9.93e-06 & 22269 & 4.0089440e+00& 1.62e+01 \\
$N(0,2)$ & Gurobi & 1.32e-17 & 2.36e-06 & 17 & 4.0089541e+00& 4.00e+01 \\
\hline
n = 200 & PPPM & 7.19e-06 & 9.92e-06 & 195 & 6.8206437e+00& 1.51e+00 \\
m = 1000 & LADMM & 9.61e-06 & 9.95e-06 & 6743 & 6.8203775e+00& 3.51e+00 \\
$\kappa$ = 0.32 & ADMM & 9.81e-06 & 3.59e-06 & 18699 & 6.8203800e+00& 1.32e+01 \\
MN & Gurobi & 2.38e-17 & 5.35e-07 & 18 & 6.8203778e+00& 4.45e+01 \\
\hline
n = 200 & PPPM & 1.50e-06 & 9.90e-06 & 152 & 3.8513292e+00& 1.32e-01 \\
m = 1000 & LADMM & 9.88e-06 & 9.96e-06 & 3918 & 3.8513145e+00& 2.19e+00 \\
$\kappa$ = 0.32 & ADMM & 9.99e-06 & 7.73e-06 & 28135 & 3.8513183e+00& 2.01e+01 \\
$\sqrt{2}t_4$ & Gurobi & 6.43e-18 & 1.16e-05 & 26 & 3.8513196e+00& 5.75e+01 \\
\hline
n = 200 & PPPM & 4.27e-06 & 7.18e-06 & 73 & 9.5835201e+00& 8.57e-02 \\
m = 1000 & LADMM & 8.52e-06 & 9.99e-06 & 4190 & 9.5833821e+00& 2.41e+00 \\
$\kappa$ = 0.32 & ADMM & 9.98e-06 & 8.92e-06 & 11156 & 9.5834183e+00& 8.04e+00 \\
Cauchy & Gurobi & 6.95e-18 & 1.31e-07 & 38 & 9.5833823e+00& 8.37e+01 \\
\hline 
\end{longtable}
\end{footnotesize}
\end{center}

From Table~\ref{T_small}, we can see that all the algorithms can solve the small size problems to the required accuracy. Gurobi is slow compared with ADMM-type methods except for the second instance. This is because the linear programming formulation of problem (\ref{RLASSO}) contains too many constraints and variables. Among the ADMM-type algorithms, the partial proximal point method (with IADMM as its subproblems solver) is more than 10 times faster than the other two algorithms. The main reason is that with the proximal term, we can use the powerful semi-smooth Newton method to solve the subproblem \eqref{Update_z} to update the variable $z.$

Now, we move on to test some large randomly generated problems. We do not consider Gurobi because it suffers from memory issue for large problems. For convenience, we only consider the first type of noise, i.e., $N(0,0.25)$ when we generate the dataset. 

\begin{center}
\begin{footnotesize}
\begin{longtable}{|c|c|ccccc|l|}
\caption{Comparison of PPPM(IADMM) and LADMM and ADMM for randomly generated rank LASSO problem.}
\label{T_large}
\\
\hline
Problem & Algorithm & Resp & Resd & Iter& Fval & Time [s] \\ \hline
\endhead
n = 1000 & PPPM & 4.37e-06 & 1.53e-06 & 51 & 1.0338039e+00& 6.75e-01 \\
m = 2000 & LADMM & 6.05e-06 & 1.00e-05 & 2193 & 1.0338041e+00& 8.70e+00 \\
$\kappa$ = 0.15 & ADMM & 5.52e-06 & 9.99e-06 & 4350 & 1.0338050e+00& 2.70e+01 \\
\hline
n = 1000 & PPPM & 5.83e-06 & 5.73e-06 & 53 & 1.0627181e+00& 1.45e+00 \\
m = 4000 & LADMM & 5.02e-06 & 9.99e-06 & 2606 & 1.0627179e+00& 3.03e+01 \\
$\kappa$ = 0.15 & ADMM & 4.56e-06 & 1.00e-05 & 6802 & 1.0627190e+00& 9.94e+01 \\
\hline
n = 2000 & PPPM & 4.30e-06 & 6.66e-06 & 52 & 8.3477708e-01& 3.91e+00 \\
m = 4000 & LADMM & 9.98e-06 & 6.53e-06 & 3469 & 8.3477689e-01& 1.73e+02 \\
$\kappa$ = 0.11 & ADMM & 9.77e-06 & 9.46e-06 & 7317 & 8.3477629e-01& 4.23e+02 \\
\hline
n = 2000 & PPPM & 3.48e-06 & 4.34e-06 & 60 & 8.5790776e-01& 9.32e+00 \\
m = 8000 & LADMM & 9.94e-06 & 9.10e-06 & 3750 & 8.5790787e-01& 3.75e+02 \\
$\kappa$ = 0.11 & ADMM & 9.88e-06 & 9.63e-06 & 31187 & 8.5791123e-01& 3.40e+03 \\
\hline
n = 4000 & PPPM & 1.88e-06 & 6.87e-06 & 56 & 6.8255971e-01& 1.94e+01 \\
m = 8000 & LADMM & 9.95e-06 & 6.99e-06 & 4465 & 6.8255971e-01& 8.12e+02 \\
$\kappa$ = 0.08 & ADMM & 2.21e-05 & 7.79e-06 & 17604 & 6.8255852e-01& 3.60e+03 \\
\hline
n = 4000 & PPPM & 2.57e-06 & 6.28e-06 & 58 & 6.9040873e-01& 2.21e+01 \\
m = 10000 & LADMM & 9.93e-06 & 7.22e-06 & 4078 & 6.9040871e-01& 9.08e+02 \\
$\kappa$ = 0.08 & ADMM & 3.32e-05 & 1.26e-05 & 14239 & 6.9040667e-01& 3.60e+03 \\
\hline
\end{longtable}
\end{footnotesize}
\end{center}

From Table~\ref{T_large}, we can see that the Algorithm~\ref{alg:PPPM} is readily more than 10 times faster than the other two algorithms. For the randomly generated datasets, the LADMM and ADMM  can return a solution of moderate accuracy for all of the instances. This is because random problems are usually well-conditioned. Apart from synthetic data, we also test on some real data
instances, which are collected from LIBSVM datasets \cite{chang2011libsvm}. We expand the features of the original data using the polynomial basis functions mentioned in \cite{huang2010predicting}. Different from randomly generated problems, the coefficient matrices $A$ for real datasets are usually ill-conditioned (with condition numbers ranging from
the order of $10^3$ to the order of more than $10^{12}$) and problem (\ref{prox-prob}) is difficult to be solved by traditional first order method. We don't show the results of certain algorithm if it reaches the maximum running time but the solution is very inaccurate. 

\begin{center}
\begin{footnotesize}
\begin{longtable}{|c|c|ccccc|l|}
\caption{Comparison of PPPM(IADMM) and LADMM and ADMM for real datasets. $(n,m)$ denotes the sample size and expanded feature size.}
\label{T_real}
\\
\hline
Problem & Algorithm & Resp & Resd & Iter& Fval & Time [s] \\ \hline
\endhead
abalone7& PPPM & 1.20e-06 & 8.53e-06 & 59 & 2.6709065e+00& 7.87e+00 \\
(4177,6435) & LADMM & 9.97e-06 & 8.17e-06 & 3706 & 2.6596874e+00& 5.77e+02 \\
$\kappa$ = 0.03 & ADMM & 6.06e-06 & 9.95e-06 & 10843 & 2.6597018e+00& 1.98e+03 \\
\hline
bodyfat7 & PPPM & 1.88e-06 & 8.48e-06 & 64 & 4.3262740e-03& 6.31e+00 \\
(252,116280) & LADMM & - & - & - & -& - \\
$\kappa$ = 0.06 & ADMM & - & - & - & -& - \\
\hline
E2006.test & PPPM & 7.74e-06 & 5.76e-06 & 62 & 3.6089303e-01& 6.50e+00 \\
(3308,150358) & LADMM & 9.99e-06 & 1.89e-06 & 655 & 3.5726244e-01& 1.15e+02 \\
$\kappa$ = 0.02 & ADMM & 1.46e-06 & 9.87e-06 & 2537 & 3.5701485e-01& 4.76e+02 \\
\hline
housing7 & PPPM & 1.47e-06 & 6.96e-06 & 66 & 7.2903775e+00& 1.20e+01 \\
(506,77520) & LADMM & 1.76e-05 & 3.32e-05 & 16685 & 7.2707283e+00& 3.60e+03 \\
$\kappa$ = 0.10 & ADMM & 1.43e-04 & 3.06e-05 & 15777 & 7.2704069e+00& 3.60e+03 \\
\hline
mpg7 & PPPM & 1.46e-06 & 5.35e-06 & 69 & 5.0256912e+00& 2.16e-01 \\
(392,3432) & LADMM & 9.15e-06 & 6.95e-06 & 1472 & 5.0191357e+00& 3.07e+00 \\
$\kappa$ = 0.10 & ADMM & 9.88e-06 & 5.25e-06 & 7906 & 5.0191178e+00& 2.14e+01 \\
\hline
pyrim5 & PPPM & 7.32e-06 & 9.29e-06 & 93 & 1.3603816e-01& 5.79e+00 \\
(74,201376) & LADMM & 4.43e-05 & 1.82e-04 & 38099 & 1.3603915e-01& 3.60e+03 \\
$\kappa$ = 0.37 & ADMM & 6.88e-04 & 1.78e-03 & 36569 & 1.3910683e-01& 3.60e+03 \\
\hline
space\_ga9 & PPPM & 3.70e-06 & 8.81e-06 & 53 & 1.6569455e-01& 3.16e+00 \\
(3107,5005) & LADMM & 4.99e-06 & 9.93e-06 & 281 & 1.6569433e-01& 2.74e+01 \\
$\kappa$ = 0.02 & ADMM & -& - & - & -& - \\
\hline
triazines4 & PPPM & 7.41e-06 & 2.67e-06 & 95 & 1.6628905e-01& 3.08e+01 \\
(186,635376) & LADMM & 2.67e-05 & 9.53e-05 & 6381 & 1.6632803e-01& 3.60e+03 \\
$\kappa$ = 0.29 & ADMM & - & - & - & -& - \\
\hline
E2006.train & PPPM & 7.96e-06 & 9.66e-06 & 85 & 3.7853164e-01& 3.76e+01 \\
(16087,150360) & LADMM & - & - & - & -& - \\
$\kappa$ = 0.01 & ADMM & 8.19e-06 & 9.85e-06 & 353 & 3.7842288e-01& 3.52e+02 \\
\hline
log1p.E2006.test & PPPM & 9.98e-06 & 4.28e-06 & 69 & 4.3580711e-01& 6.68e+01 \\
(3308,4272226) & LADMM & 8.32e-04 & 4.17e-03 & 3027 & 4.1710211e-01& 3.60e+03 \\
$\kappa$ = 0.10 & ADMM & - & - & - & -& - \\
\hline
log1p.E2006.train & PPPM & 8.92e-06 & 1.72e-06 & 71 & 4.1277825e-01& 2.40e+02 \\
(16087,4272227) & LADMM & - & - & - & -& - \\
$\kappa$ = 0.05 & ADMM & - & - & - & -& - \\
\hline
\end{longtable}
\end{footnotesize}
\end{center}

From Table~\ref{T_real}, we can see that Algorithm~\ref{alg:PPPM} is the only solver that can solve all the problems to the required accuracy. Moreover, Algorithm~\ref{alg:PPPM} is much more efficient than the other two algorithms. For the instances {\tt housing7} and {\tt pyrim5}, Algorithm~\ref{alg:PPPM} is more than 100 times faster than the other two algorithms. This implies that our algorithm is less affected by ill-conditioning of the dataset. This verifies the efficiency and robustness of the partial proximal point method
(with IADMM as its subproblems solver) for solving non-semi-strongly convex
problems of the form \eqref{eq-prob}.

\section{Conclusion}\label{Sec-conc}
We have proposed an adaptive ADMM which can adjust the penalty parameters
adaptively with a large degree of freedom. Various types of convergence results for IADMM have been established under the semi-strongly convex condition. We have also proposed a partial proximal point method (together with IADMM as its subproblems solver)
to solve problems without semi-strongly convexity. Numerical experiments show that the convergence of IADMM with self-adaptive parameters 
adjustment is insensitive to the initial parameter chosen as compared to the fixed-parameter ADMM. Also, the partial proximal point method is much more efficient compared with other ADMM-type methods. There are further research questions that we can explore, and these include analyzing the convergence rate of partial proximal point method and applying this method to solve other problems with two non-smooth functions.

\section*{Acknowledgement}
We thank the reviewers and Associate Editor for many helpful suggestions
to improve the quality of the paper.

\bibliographystyle{abbrv}
\bibliography{IADMM}

\appendix

\section{Proof details}
\subsection{Proof of Lemma~\ref{onestep}}

\begin{proof}
From the optimality conditions in step 1 and 2, we have that 
\begin{align}
& -\Big(B^\top\lambda^k+\beta_kB^\top(By^{k+1}+Cz^k-b)+P_k(y^{k+1}-y^k)\Big)
\in \partial f(y^{k+1})
\label{eq-opt-f}
\\
& -\Big(C^\top \lambda^k+\beta_kC^\top(By^{k+1}+Cz^{k+1}-b)+Q_k(z^{k+1}-z^k)\Big)
\in \partial g(z^{k+1}).
\label{eq-opt-g}
\end{align}
From \eqref{eq-opt-f} and the convexity of $f$, we have that
\begin{align}\label{six}
&f(y^{k+1})-f(y)\leq
\inprod{ B^\top \lambda^k+\beta_kB^\top (By^{k+1}+Cz^k-b)+P_k(y^{k+1}-y^k)}{y-y^{k+1}}\notag \\
&=\inprod{ \lambda^k+\beta_k(By^{k+1}+Cz^k-b)}{By-By^{k+1}}
+\eta_{P_k}(y,y^k,y^{k+1}).
\end{align}
Similarly, from \eqref{eq-opt-g} and \eqref{gstr}, we have that
\begin{align}\label{sev}
&g(z^{k+1})-g(z)
\notag \\
&\leq \inprod{\lambda^k+\beta_k(By^{k+1}+Cz^{k+1}-b)}{Cz-Cz^{k+1}}
+\eta_{Q_k}(z,z^k,z^{k+1})-\frac{\sigma_g}{2}\ll z^{k+1}-z\ll^{2}\notag \\
&=\< \lambda^k+\beta_k(By^{k+1}+Cz^k-b),Cz-Cz^{k+1}\>+\eta_{\beta_kC^\top C+Q_k}(z,z^k,z^{k+1})-\frac{\sigma_g}{2}\ll z^{k+1}-z\ll^{2}.
\end{align}
From (\ref{six}), (\ref{sev}) we have that
\begin{align}\label{eigh}
&\F(x^{k+1})-\F(x)\notag \\
&\leq \inprod{\lambda^k+\beta_k(By^{k+1}+Cz^k-b)}{b-\A x^{k+1}}
 +\eta_{\beta_kC^\top C+Q_k}(z,z^k,z^{k+1})+\eta_{P_k}(y,y^k,y^{k+1})
 \notag\\
 & \quad -\frac{\sigma_g}{2}\ll z-z^{k+1}\ll^{2}
 \notag \\
&=\< \lambda^k+\beta_k(\A x^{k+1}-b),b-\A x^{k+1}\>+\beta_k\< C(z^k-z^{k+1}),b-\A x^{k+1}\>\notag\\
&\quad +\eta_{\beta_k C^\top C+Q_k}(z,z^k,z^{k+1})
+\eta_{P_k}(y,y^k,y^{k+1})-\frac{\sigma_g}{2}\ll z-z^{k+1}\ll^2
\notag \\
&=\< \lambda^{k+1},b-\A x^{k+1}\>+(\gamma-1)\beta_k\ll \A x^{k+1}-b\ll^2+\beta_k\< C(z^k-z^{k+1}),b-\A x^{k+1}\>
\notag \\
& \quad +\eta_{\beta_k C^\top C+Q_k}(z,z^k,z^{k+1})
  +\eta_{P_k}(y,y^k,y^{k+1})-\frac{\sigma_g}{2}\ll z-z^{k+1}\ll^2,
\end{align}
where we have used step 3 to get the last equality. 
From (\ref{eigh}), we have 
\begin{align}\label{ei1}
&\L(x^{k+1},\lambda)-\L(x,\lambda)\notag \\
&\leq \< \lambda^{k+1}-\lambda,b-\A x^{k+1}\>+(\gamma-1)\beta_k\ll \A x^{k+1}-b\ll^2+\beta_k\< C(z^k-z^{k+1}),b-\A x^{k+1}\>\notag \\
&\quad 
+\eta_{\beta_k C^\top C+Q_k}(z,z^k,z^{k+1})+\eta_{P_k}(y,y^k,y^{k+1})-\frac{\sigma_g}{2}\ll z-z^{k+1}\ll^2\notag \\
&=\Big\langle \lambda^{k+1}-\lambda,\frac{\lambda^k-\lambda^{k+1}}{\beta_k \gamma}\Big\rangle+(\gamma-1)\beta_k\ll \A x^{k+1}-b\ll^2+\beta_k\< C(z^k-z^{k+1}),b-\A x^{k+1}\>
\notag\\
&\quad +\eta_{\beta_k C^\top C+Q_k}(z,z^k,z^{k+1})
+\eta_{P_k}(y,y^k,y^{k+1})-\frac{\sigma_g}{2}\ll z-z^{k+1}\ll^2.
\end{align}

Now, we need to estimate $\beta_k\< C(z^k-z^{k+1}),b-\A x^{k+1}\>$. 
From \eqref{eq-opt-g}, we know that 
\begin{align}
&-C^\top \lambda^k-\beta_k C^\top (\A x^{k+1}-b)-Q_k(z^{k+1}-z^k)\in \partial g(z^{k+1})\notag \\
&-C^\top \lambda^{k-1}-\beta_{k-1}C^\top (\A x^k-b)-Q_{k-1}(z^k-z^{k-1})\in \partial g(z^k)\notag
\end{align}
Combining the above two equations together with the strongly convexity of $g$, we get 
\begin{align} 
\Big\langle
\begin{array}{c}
C^\top (\lambda^{k-1}-\lambda^k)-\beta_k C^\top (\A x^{k+1}-b)-Q_k(z^{k+1}-z^k)
\\
+\beta_{k-1}C^\top (\A x^k-b)+Q_{k-1}(z^k-z^{k-1})
\end{array},z^{k+1}-z^k \Big\rangle
\geq \sigma_g \norm{z^k-z^{k+1}}^2,\notag
\end{align}
which, together with step 3, implies that 
\begin{align}
& \sigma_g \norm{z^k-z^{k+1}}^2 
 \notag \\
&  \leq \< \lambda^{k-1}+\beta_{k-1}(\A x^k-b)-\lambda^k-\beta_k(\A x^{k+1}-b),C(z^{k+1}-z^k)\>\notag \\
&\quad +\< -Q_k(z^{k+1}-z^k)+Q_{k-1}(z^k-z^{k-1}),z^{k+1}-z^k\> \notag \\
&=\< (1-\gamma)\beta_{k-1}(\A x^k-b)-\beta_k(\A x^{k+1}-b),C(z^{k+1}-z^k)\>\notag \\
&\quad +\< -Q_{k-1}(z^{k+1}-z^k)+Q_{k-1}(z^k-z^{k-1}),z^{k+1}-z^k\>-(\beta_k-\beta_{k-1})\ll z^{k+1}-z^k\ll^2_Q.\notag 
\end{align}
The above inequality implies that
\begin{align}
&\beta_k \< \A x^{k+1}-b,C(z^{k+1}-z^k)\>
\notag \\
& \leq
 (\gamma-1)\< \beta_{k-1}(b-\A x^k),C(z^{k+1}-z^k)\>
 -\ll z^{k+1}-z^k\ll^2_{Q_{k}}\notag \\
&\quad +\beta_{k-1}\< Q(z^k-z^{k-1}),z^{k+1}-z^k\>-\sigma_g\ll z^k-z^{k+1}\ll^2\notag \\
&\leq \frac{(\gamma-1)\beta_{k-1}^2}{2\gamma\beta_k}\ll \A x^k-b\ll^2
+\frac{(\gamma-1)\gamma\beta_k}{2}\ll C(z^{k+1}-z^k)\ll^2-\ll z^{k+1}-z^k\ll^2_{Q_k}\notag \\
&\quad +\frac{\beta_{k-1}^2}{2\beta_k}\ll z^k-z^{k-1}\ll^2_Q+\frac{\beta_k}{2}\ll z^{k+1}-z^k\ll^2_Q-\sigma_g\ll z^k-z^{k+1}\ll^2.\notag\\
&=\frac{(\gamma-1)\beta_{k-1}^2}{2\gamma\beta_k}\ll \A x^k-b\ll^2
+\frac{(1-\delta)\beta_k}{2}\ll C(z^{k+1}-z^k)\ll^2+\frac{\beta_{k-1}^2}{2\beta_k}\ll z^k-z^{k-1}\ll^2_Q-\frac{\beta_k}{2}\ll z^{k+1}-z^k\ll^2_Q\notag \\
&\quad -\sigma_g\ll z^k-z^{k+1}\ll^2,\notag
\end{align}
In the above, we use the fact that $\gamma(\gamma-1) = 1-\delta$.
Now, we plug the above inequality into (\ref{ei1}), we get 
\begin{align}
&\L(x^{k+1},\lambda)-\L(x,\lambda)\notag\\
&\leq \Big\langle  \lambda^{k+1}-\lambda,\frac{\lambda^k-\lambda^{k+1}}{\beta_k\gamma}\Big\rangle+(\gamma-1)\beta_k\ll \A x^{k+1}-b\ll^2
+\frac{(\gamma-1)\beta_{k-1}^2}{2\gamma\beta_k}\ll \A x^k-b\ll^2+\eta_{P_k}(y,y^k,y^{k+1})\notag \\
&\quad +\eta_{\beta_k C^\top C+Q_k}(z,z^k,z^{k+1})
+\frac{(1-\delta)\beta_k}{2}\ll C(z^{k+1}-z^k)\ll^2+\frac{\beta_{k-1}^2}{2\beta_k}\ll z^k-z^{k-1}\ll^2_Q-\frac{\beta_k}{2}\ll z^{k+1}-z^k\ll^2_Q\notag \\
&\quad -\sigma_g\norm{ z^k-z^{k+1}}^2
-\frac{\sigma_g}{2}\norm{ z-z^{k+1}}^2\notag
\end{align}
\begin{align}\label{hehehe}
&\leq \Big\langle \lambda^{k+1}-\lambda,\frac{\lambda^k-\lambda^{k+1}}{\beta_k\gamma}
\Big\rangle+(\gamma-1)\beta_k\ll \A x^{k+1}-b\ll^2
+\frac{(\gamma-1)\beta_{k-1}^2}{2\gamma\beta_k}\ll \A x^k-b\ll^2
+\eta_{P_k}(y,y^k,y^{k+1})\notag \\
&\quad +\xi_{\beta_k C^\top C+Q_k}(z,z^k,z^{k+1}) -
\frac{\delta\beta_k}{2}\norm{C(z^{k+1}-z^k)}^2
+\frac{\beta_{k-1}^2}{2\beta_k}\ll z^k-z^{k-1}\ll^2_Q-\beta_k\ll z^{k+1}-z^k\ll^2_Q
\notag\\
&\quad -\sigma_g\ll z^k-z^{k+1}\ll^2-\frac{\sigma_g}{2}\ll z-z^{k+1}\ll^2.
\end{align}
Note that from step 4, we can derive that
\begin{eqnarray*}
\beta_k\left(\xi_{\beta_k C^\top C+Q_k}(z,z^k,z^{k+1})-\frac{(1-\epsilon)\sigma_g}{2}\ll z-z^{k+1}\ll^2\right)\leq \frac{\beta_k^2}{2}\ll z-z^k\ll^2_{C^\top C+Q}-\frac{\beta_{k+1}^2}{2}\ll z-z^{k+1}\ll^2_{ C^\top C+Q}.
\end{eqnarray*}
Multiply (\ref{hehehe}) by $\beta_k$ and use the above inequality, we obtain that
\begin{align}
&\beta_k\left(\L(x^{k+1},\lambda)-\L(x,\lambda)\right)\notag \\
&\leq \frac{1}{\gamma}\< \lambda^{k+1}-\lambda, \lambda^k-\lambda^{k+1}\>+(\gamma-1)\beta_k^2\ll \A x^{k+1}-b\ll^2+\frac{(\gamma-1)\beta_{k-1}^2}{2\gamma}\ll \A x^k-b\ll^2
\notag\\
&\quad +\beta_k\eta_{P_k}(y,y^k,y^{k+1}) 
+\frac{\beta_k^2}{2}\ll z-z^k\ll^2_{C^\top C+Q}-\frac{\beta_{k+1}^2}{2}\ll z-z^{k+1}\ll^2_{C^\top C+Q} 
-\frac{\delta\beta_k^2}{2}\norm{C(z^{k+1}-z^k)}^2
\notag \\
&\quad +\frac{\beta_{k-1}^2}{2}\ll z^k-z^{k-1}\ll^2_Q
-\beta_k^2\ll z^{k+1}-z^k\ll^2_Q -\sigma_g\beta_k\ll z^k-z^{k+1}\ll^2-\frac{\epsilon \sigma_g\beta_k}{2}\ll z-z^{k+1}\ll^2 
\notag\\
&=\frac{1}{\gamma}\xi(\lambda,\lambda^k,\lambda^{k+1})
-\frac{(2-\gamma)\beta_k^2}{2}\norm{\A x^{k+1}-b}^2+\frac{(\gamma-1)\beta_{k-1}^2}{2\gamma}\ll A x^{k}-b\ll^2
\notag \\
&\quad +\beta_k\eta_{P_k}(y,y^k,y^{k+1}) +\frac{\beta_k^2}{2}\ll z-z^k\ll^2_{C^\top C+Q}-\frac{\beta_{k+1}^2}{2}\ll z-z^{k+1}\ll^2_{C^\top C+Q} 
-\frac{\delta\beta_k^2}{2}\norm{C(z^{k+1}-z^k)}^2
\notag \\
&\quad +\frac{\beta_{k-1}^2}{2}\ll z^k-z^{k-1}\ll^2_Q
-\beta_k^2\ll z^{k+1}-z^k\ll^2_Q
-\sigma_g\beta_k\ll z^k-z^{k+1}\ll^2-\frac{\epsilon \sigma_g\beta_k}{2}\ll z-z^{k+1}\ll^2\notag 
\end{align}
where we have used the fact that
$\inprod{\lam_{k+1}-\lam}{\lam_k-\lam_{k+1}} = \frac{1}{2}\norm{\lam-\lam_k}^2
-\frac{1}{2}\norm{\lam-\lam_{k+1}}^2 -\frac{1}{2}\norm{\lam_k-\lam_{k+1}}^2$ 
and 
$\lambda^{k+1}-\lambda^{k}=\gamma\beta_k (\A x^{k+1}-b).$
Note that since $\gamma\in (1,\frac{1+\sqrt{5}}{2}),$ 
$\delta = 1+\gam-\gam^2 > 0$. Using the identity, 
$\frac{\gamma-1}{2\gamma}= \frac{(2-\gamma)}{2}-\frac{\delta}{2\gamma}$, 
we deduce that 
\begin{align}
&\beta_k\left( \L(x^{k+1},\lambda)-\L(x,\lambda)\right) 
+\frac{\delta \beta_{k-1}^2}{2\gamma}\ll \A x^{k}-b\ll^2
\notag \\
&\leq \frac{1}{\gamma}\xi (\lambda,\lambda^k\lambda^{k+1})+
\frac{(2-\gamma)\beta_{k-1}^2}{2}\ll \A x^k-b\ll^2
-\frac{(2-\gamma)\beta_{k}^2}{2}\norm{ \A x^{k+1}-b}^2
\notag \\
&\quad +\beta_k\eta_{P_k}(y,y^k,y^{k+1}) +\frac{\beta_k^2}{2}\ll z-z^k\ll^2_{C^\top C+Q}-\frac{\beta_{k+1}^2}{2}\ll z-z^{k+1}\ll^2_{C^\top C+Q}
 -\frac{\delta\beta_k^2}{2}\norm{C(z^{k+1}-z^k)}^2
\notag \\
&\quad +\frac{\beta_{k-1}^2}{2}\ll z^k-z^{k-1}\ll^2_Q
-\beta_k^2\ll z^{k+1}-z^k\ll^2_Q-\sigma_g\beta_k\ll z^k-z^{k+1}\ll^2-\frac{\epsilon \sigma_g\beta_k}{2}\ll z-z^{k+1}\ll^2.\notag 
\end{align}
From here, one can readily get  the required inequality in Lemma~\ref{onestep}.
\end{proof}

\subsection{Proof of Lemma~\ref{useful}}

\begin{proof}
Because $(x^*,\lambda^*)$ is a KKT solution, we have that $0\in \partial_x \L(x^*,\lambda^*)$. Since $\L(x,\lambda^*)$ is a convex function of $x$, we then have that 
\begin{equation}\label{telv}
\L(x^*,\lambda^*)\leq \L(x,\lambda^*)\ {\rm for\ any}\ x, 
\end{equation}
from which we get
\begin{equation}\label{thirt}
-\inprod{\lambda^*}{\A x^k-b} \leq \F(x^k)-\F(x^*).
\end{equation} 
Consider all $\lam\in \R^m$ such that $\ll \lambda\ll \leq \ll \lambda^*\ll+1$ in $\L(x^k,\lambda)-\L(x^*,\lambda)\leq h(k)D(\lambda)$, we have
\begin{equation}\label{fourt}
\F(x^k)-\F(x^*)+(\ll\lambda^*\ll+1)\ll \A x^k-b\ll\leq h(k)\max_{\ll \lambda\ll\leq \ll \lambda^*\ll+1}D(\lambda).
\end{equation}
Using (\ref{thirt}) in (\ref{fourt}), we get 
\begin{equation}\label{fift}
\ll \A x^k-b\ll\leq h(k)\max_{\ll \lambda\ll\leq \ll \lambda^*\ll+1}D(\lambda)=O(h(k)).
\end{equation}
Now, using (\ref{fift}) in (\ref{thirt}) and (\ref{fourt}) respectively, we get
$| \F(x^k)-\F(x^*)|=O(h(k)).$
\end{proof}

\subsection{Proof of Lemma~\ref{energy1}}
\begin{proof}
Substitute $(x^*,\lambda^*)$ into (\ref{onestepineq}), we get the following long inequality

\begin{align}\label{longineq}
&\beta_k\( \L(x^{k+1},\lambda^*)-\L(x^*,\lambda^*) \)+\overbrace{\frac{\delta\beta^2_{k-1}}{2\gamma}\|\A x^k-b\|^2}^{1}+\overbrace{\frac{\delta\beta^2_k}{2}\|C(z^{k+1}-z^k)\|^2}^{0}\notag \\
&+\overbrace{\frac{\beta_k}{2}\| y^k-y^{k+1} \|^2_{P_k}}^{0}+\frac{\beta_k^2}{2}\|z^k-z^{k+1}\|^2_Q+\frac{1}{2\gamma}\|\lambda^*-\lambda^{k+1}\|^2+\frac{(2-\gamma)\beta^2_{k}}{2}\|\A x^{k+1}-b\|^2\notag \\
&+\frac{\beta^2_{k+1}}{2}\|z^*-z^{k+1}\|^2_{C^\top C+Q}+\frac{\beta_k^2}{2}\|z^{k+1}-z^k\|^2_Q+\overbrace{\frac{\beta_{k+1}}{2}\| y^*-y^{k+1} \|^2_{P_{k+1}}}^{0}\notag \\
&\leq \frac{1}{2\gamma} \|\lambda^*-\lambda^k\|^2+\frac{(2-\gamma)\beta_{k-1}^2}{2}\| \A x^k-b\|^2+\frac{\beta_k^2}{2}\|z^*-z^k\|^2_{C^\top C+Q}+\frac{\beta_{k-1}^2}{2}\|z^k-z^{k-1}\|^2_Q\notag \\
&+\overbrace{\frac{\beta_k}{2}\|y^*-y^k\|^2_{P_k}}^{0}-\overbrace{\sigma_g \beta_k \|z^k-z^{k+1}\|^2}^{2}-\overbrace{\frac{\epsilon\sigma_g\beta_k}{2}\|z^*-z^{k+1}\|^2}^{3}
\end{align}
Now, we apply several operations to the above inequality: 1, ignore terms under ``0" since $P=0$; 2, move the term under ``1" to the right hand side; 3, move the term under ``2" to the left hand side and apply $\|z^{k+1}-z^k\|^2_Q/\lambda_{\max}(Q)\leq \|z^k-z^{k+1}\|^2$; 4, move one half of ``4" to the left hand side and apply $\|z^{k+1}-z^*\|^2\leq \|z^{k+1}-z^*\|^2_{C^\top C+Q}/\lambda_{\max(C^\top C+Q)}.$ After all these operations, we will get the inequality (\ref{ieqlong}).
\end{proof}

\subsection{Proof of Lemma~\ref{wy}}
\begin{proof}
From step 2 and step 3 in IADMM, we have that
\begin{equation}
0=\nabla g(z^{k+1})+C^\top \lambda^{k+1}+(1-\gamma)\beta_k C^\top (\A x^{k+1}-b)+Q_k(z^{k+1}-z^k). \notag 
\end{equation}
Since $(y^*,z^*,\lambda^*)$ is a KKT solution, we have $0=\nabla g(z^*)+C^\top \lambda^*.$ Combining these two equations together with the Lipschitz continuity of $\nabla g$, we have
\begin{align}\label{thir5}
&\ll C^\top (\lambda^{k+1}-\lambda^*)+(1-\gamma)\beta_k C^\top (\A x^{k+1}-b)+Q_k(z^{k+1}-z^k)\ll^2=\ll \nabla g(z^{k+1})-\nabla g(z^*)\ll^2\notag \\
&\leq L_g^2\ll z^{k+1}-z^*\ll^2.
\end{align}
For $0<\alpha<\frac{1}{2}$, by using the inequality $\norm{u+v+w}^2 \geq (1-2\alpha)\norm{u}^2-\frac{1}{\alpha}\norm{v}^2-\frac{1}{\alpha}\norm{w}^2$, we have that
\begin{align}
&\ll C^\top (\lambda^{k+1}-\lambda^*)+(1-\gamma)\beta_k C^\top (\A x^{k+1}-b)+Q_k(z^{k+1}-z^k)\ll^2\notag \\
&\geq 
(1-2\alpha)\ll C^\top (\lambda^{k+1}-\lambda^*)\ll^2
-\frac{1}{\alpha}\ll (1-\gamma) \beta_k C^\top(\A x^{k+1}-b)\ll^2
-\frac{1}{\alpha}\ll Q_k(z^{k+1}-z^k)\ll^2\notag \\
&\geq (1-2\alpha)\lambda_{\min}(CC^\top)\ll\lambda^{k+1}-\lambda^*\ll^2
-\frac{1}{\alpha}\lambda_{\max}(CC^\top)(1-\gamma)^2\beta_k^2\ll \A x^{k+1}-b\ll^2 -\frac{1}{\alpha}\ll Q_k(z^{k+1}-z^k)\ll^2. \notag
\end{align}
Plug this into (\ref{thir5}), we get
\begin{align}\label{thir6}
&(1-2\alpha)\lambda_{\min}(CC^\top) \ll \lambda^{k+1}-\lambda^*\ll^2\notag \\
&\leq \frac{1}{\alpha}\lambda_{\max}(CC^\top)(1-\gamma)^2\beta_k^2\ll \A x^{k+1}-b\ll^2
+ \frac{1}{\alpha} \ll Q_k(z^{k+1}-z^k)\ll^2+L_g^2\ll z^{k+1}-z^*\ll^2\notag \\
&\leq \frac{1}{\alpha}\lambda_{\max}(CC^\top)(1-\gamma)^2\beta_k^2\ll \A x^{k+1}-b\ll^2+\frac{1}{\alpha}\lambda_{\max}(Q)\beta_k^2\ll z^{k+1}-z^k\ll^2_Q+L_g^2\ll z^{k+1}-z^*\ll^2.\notag
\end{align}
This completes the proof.
\end{proof}

\end{document}